\documentclass[10pt, a4paper,reqno]{amsart}

\usepackage{mathtools}
\usepackage{amsmath,amsthm, amsfonts, amsrefs, amssymb}
\mathtoolsset{showonlyrefs}

\usepackage[colorlinks, linkcolor = black, citecolor = black, filecolor = black, urlcolor = blue]{hyperref}
\usepackage{graphicx}
\usepackage{bbm}
\usepackage{tikz}
\usepackage{tikz-cd}
\usepackage{hyperref}
\usepackage{todonotes}
\usepackage{mathrsfs}
\usepackage[margin=2cm]{geometry}
\usepackage{enumitem}
\usepackage{cancel}

\numberwithin{equation}{section}
\theoremstyle{plain}
\newtheorem{thm}{Theorem}[section]
\newtheorem{lemma}[thm]{Lemma}
\newtheorem{col}[thm]{Corollary}
\newtheorem{prop}[thm]{Proposition}

\theoremstyle{definition}
\newtheorem{defn}[thm]{Definition}

\newtheorem{remark}[thm]{Remark}

\newcommand{\nc}{\newcommand}
\nc{\sgn}{\mathrm{sgn}}
\nc{\nothing}{\varnothing}
\nc{\Ab}{\mathrm{Ab}}
\nc{\Graph}{\mathbf{Graph}}
\nc{\FinSet}{\mathbf{FinSet}}
\nc{\Vect}{\mathbf{Vect}}
\nc{\Top}{\mathbf{Top}}
\nc{\Cat}{\mathbf{CAT}}
\nc{\Set}{\mathbf{Set}}
\nc{\Rel}{\mathbf{Rel}}
\nc{\Grp}{\mathbf{Grp}}
\nc{\AbGrp}{\mathbf{AbGrp}}
\nc{\cT}{\mathcal T}
\nc{\cG}{\mathcal G}
\nc{\cF}{\mathcal F}
\nc{\cK}{\mathcal{K}}
\nc{\cE}{\mathcal E}
\nc{\cP}{\mathcal P}
\nc{\cM}{\mathcal M}
\nc{\cC}{\mathcal C}
\nc{\cB}{\mathcal B}
\nc{\cS}{\mathcal S}
\nc{\cD}{\mathcal D}
\nc{\cA}{\mathcal A}
\nc{\PP}{\mathbb P}
\nc{\cU}{\mathcal U}
\nc{\Mod}{\operatorname{Mod}}
\nc{\Aut}{\operatorname{Aut}}
\nc{\del}{\partial}
\nc{\inter}{\mathrm{o}}
\nc{\close}[1]{\overline{#1}}
\nc{\pderiv}[2]{\frac{\partial #1}{\partial #2}}
\nc{\tr}{\operatorname{tr}}
\nc{\dd}{\mathrm{d}}
\renewcommand{\epsilon}{\varepsilon}

\nc{\RP}{\R\mathrm{P}}

\newcommand{\R}{\mathbb{R}}
\renewcommand{\P}{\mathbb{P}}
\newcommand{\E}{\mathbb{E}}
\newcommand{\N}{\mathbb{N}}

\newcommand{\T}{\mathbb{T}}

\newcommand{\inv}{^{-1}}

\nc{\ev}{\mathrm{ev}}
\nc{\Nat}{\mathrm{Nat}}

\newcommand{\ra}{\rightarrow}

\newcommand{\diver}{\mathrm{div}}
\nc{\vect}[1]{\mathbf{#1}} 
\nc{\im}{\mathrm{im}}
\nc{\weakra}{\rightharpoonup}
\nc{\Cof}{\mathrm{Cof}}
\nc{\innprod}[2]{\left\langle #1, #2 \right\rangle}
\nc{\norm}[1]{\left\| #1 \right\|}
\nc{\abs}[1]{\left\lvert #1 \right\rvert}
\nc{\e}[1]{ \mathbb E \left[ #1 \right] }
\renewcommand{\d}{\mathrm{d}}
\newcommand{\loc}{\mathrm{loc}}
\nc{\vp}{\varphi}

\allowdisplaybreaks

\usepackage[dvipsnames]{xcolor}
\title[FSOP and WTP  for porous  media equations with nonlinear conservative noise]{Finite speed of propagation and  waiting time phenomena for stochastic porous media equations with nonlinear conservative noise}
\author{G\"unther Gr\"un}
\address{Department Mathematik\\ Friedrich--Alexander--Universit\"at Erlangen--N\"urnberg\\ Cauerstra{\ss}e 11\\
91058 Erlangen\\ Germany}
\email{gruen@math.fau.de}
\author{Max Sauerbrey}
\address{Max Planck Institute for Mathematics in the Sciences\\
Inselstr. 22 \\ 04103 Leipzig \\ Germany.} \email{maxsauerbrey97@gmail.com}
\author{Joshua Utley}
\address{Department Mathematik\\ Friedrich--Alexander--Universit\"at Erlangen--N\"urnberg\\ Cauerstra{\ss}e 11\\
91058 Erlangen\\ Germany}
\email{utley1999@gmail.com}
\date{\today}

\subjclass[2020]{35R35, 35K65, 35R60, 60H15, 76S05}

\keywords{Porous media equation, conservative noise, qualitative properties, support propagation}

\begin{document}

\begin{abstract}
    Starting from localized energy estimates, we prove finite speed of propagation for kinetic solutions to stochastic porous media equations with nonlinear conservative noise, the existence and uniqueness of which has recently been established. In particular, we propose a novel iteration technique which allows us to obtain  a Stampacchia-type inequality involving one single integral quantity, despite the possibly different scaling behaviors of the porous media and the noise term. This allows us to apply stochastic filtering arguments developed for the case of linear source-type noise. Using related ideas, we identify flatness conditions on  initial data which guarantee locally the occurrence of a waiting time phenomenon, i.e., the onset of forward propagation of the solution's support is locally delayed. The condition for the latter matches the one for the deterministic porous media equation up to a logarithmic correction in the case of critical nonlinearity in the noise, but it requires more and more flatness of the initial data as the nonlinearity tends towards linear conservative noise. This is in line with the expected behavior: In the case of linear conservative noise, instantaneous forward motion is possible due to the effects of stochastic transport, no matter how flat the initial profile is.       
\end{abstract}

\maketitle

\section{Introduction}

This paper is concerned with qualitative properties of the support evolution for non-negative solutions to stochastic porous media equations of the form 
\begin{equation}\label{Eq_SPME_intro}\tag{sPME}
	\begin{cases}
		\d u \,=\, \Delta (u^m) \,\d t \,-\, \diver( \sigma(u) \circ \d W ),\\
		u(0) \,=\, u_0,\end{cases}
\end{equation}
where $W$ is spatially colored Wiener process, and we assume
\begin{equation}\label{Eq_m}
	m\in (1,\infty).
\end{equation}

We investigate two characteristic qualitative phenomena of deterministic porous media equations: finite speed of propagation and the waiting time phenomenon. Put concisely, a solution has finite speed of propagation (FSOP) when the following is true: Given any open ball, $B$, whose closure is disjoint from the initial support, a positive time must elapse almost-surely before mass enters $B$. A waiting time phenomenon (WTP) occurs when the following statement is true: for some open ball $B'$ such that $B'\, \cap \, \mathrm{supp}(u_0)=\emptyset$, yet $\overline{B'}\,\cap\, \mathrm{supp}(u_0)\neq\emptyset$, a positive time elapses almost-surely before mass enters $B'$. In this case, we say that $u$ exhibits a WTP in $\partial B'\cap \mathrm{supp}(u_0)$. Note that these notions only consider the forward expansion of the support, but do not handle the possibility of backwards motion or contraction.   

Such qualitative properties have been studied extensively in degenerate parabolic equations, and some results have been proven for stochastic porous media equations with a specific choice of the noise (cf. Subsection \ref{SS:state}). Until this work, however, understanding of these phenomena for \eqref{Eq_SPME_intro} has remained open. One insight needed to handle \eqref{Eq_SPME_intro} is an understanding of how the growth of the nonlinearity $\sigma$ influences the support expansion. In order to treat this, we make the following prototypical choice of noise coefficients:
\begin{equation}\label{Eq_n}
	\sigma(u) \,=\, u^n, \qquad  n\in (1,(m+1)/2].
\end{equation}

Under this assumption, our main result, Theorem \ref{Thm_fsop}, asserts the following: The solution to \eqref{Eq_SPME_intro} has a finite speed of propagation. Moreover, if $u_0$ is ``sufficiently flat" in a neighborhood of a point on the boundary of its support, then the solution to \eqref{Eq_SPME_intro} admits a forward waiting time phenomenon there, where the condition of ``sufficient flatness" is dominated by the nonlinear noise term.

By ``the solution to \eqref{Eq_SPME_intro}", we are referring to the unique stochastic kinetic solution introduced in \cite{Fehrman_Gess_ARMA} based on ideas from stochastic conservation laws \cite{LPS}, and following a long line of research on the well-posedness of stochastic porous media equations with other noise terms \cite{BPR_non_neg, BPR_strong_sol_PME, wang1, wang2, BPR_criticality,Gess_strong,FG_pathwise, FG_pathwiseII, Dirr_Gruen_Grillmeier, DGG_white,BVW_degenerate,DHV_16, Gess_Hof,DGG_Entropy,spme_book}. That, as shown in \cite{Fehrman_Gess_ARMA}, the  kinetic solution  is uniquely selected by a viscous regularization procedure  plays an important role in our proof, which is further elaborated  in Subsection \ref{SS_proof} below.

\subsection{Support propagation of stochastic porous media equations: State of the art}\label{SS:state}
In the deterministic case, support evolution sets apart solutions to degenerate porous media equations, i.e, $\partial_t u =\Delta (u^m)$ with $m>1$, from the heat equation ($m=1$) and fast diffusion equations $m<1$. Indeed, while the fundamental solution to the heat equation 
\begin{equation}\label{eqn1}
\Phi(t,x) \,=\, 
\frac{1}{(4\pi t)^{d/2}} \mathrm{exp} \biggl( \frac{- |x|^2}{4t}\biggr),
\end{equation}
is everywhere strictly positive at any time $t>0$, the Barenblatt solutions
\begin{equation}\label{Eq_Barenblatt}
U(t,x) \,=\, t^{\frac{-d}{d (m-1) +2}}\biggl(
C \,-\,\frac{m-1}{2m (d (m-1)+2) } |x|^2t^{\frac{-2}{d (m-1) +2}}
\biggr)_+^{1/(m-1)},\qquad C\in (0,\infty),
\end{equation}
to the porous media equation have FSOP. Using comparison principles and appropriately rescaled Barenblatt solutions, FSOP may consequently be established for sufficiently regular solutions to porous media equations with compactly supported initial data (see \cite{ Aronson70,Knerr77,VazquezPME} and the references therein). 

Whether or not this characteristic property of the porous media equation persists under stochastic perturbations as in \eqref{Eq_SPME_intro} is therefore also a mathematically natural question. Considering that comparison to a Barenblatt profile is only feasible in the deterministic setting, the investigation of FSOP for  stochastic porous media equations requires different and more elaborate methods of proof. This question was first addressed for solutions to the stochastic porous media equation
\begin{equation}\label{Eq_mult_noise_SPME}
\begin{cases}
\d u \,=\, \Delta (u^m)\, \d t \,+\,  u\, \d W,\\
	u(0) \,=\, u_0, 
\end{cases}
\end{equation}
in \cite{Barbu_Rock_FSOP,gess_fsop,Fischer_Grun_FSOP}, where it is shown under various assumptions and with different techniques that solutions have FSOP -- both for It\^o  and for Stratonovich noise. The first result in this direction, \cite{Barbu_Rock_FSOP}, relied on the rescaling transformation $v = \mathrm{exp} (-W )u$, which eliminates the stochastic integral in \eqref{Eq_mult_noise_SPME} at the expense of modulating the spatial differential and introducing an additional deterministic term emerging as an It\^o correction. If the linear multiplicative noise is instead taken to be of Stratonovich type, no extra term emerges, and an adaption of the hole filling method leads to an improved result in \cite{gess_fsop}, where also non-Brownian drivers are allowed by means of rough path theory and upper bounds on propagation rates are obtained. The authors of \cite{Fischer_Grun_FSOP} instead used a novel stochastic energy method which does not need any transformation of the equation \eqref{Eq_mult_noise_SPME}. The technique, which we elaborate on further in Subsection \ref{SS_proof}, takes inspiration from arguments in the deterministic case that are known to deliver optimal results there (cf. \cite{DalPasso_Giacomelli_Grun_PISA, HilKar, Grun2001b, GrunPoincare, Giacomelli_Grun_WTP_general, AnsiniGiacomelli2004, Utley_TFPME}). The latter was also used in \cite{GrillmeierDissertation} to prove FSOP for solutions to the stochastic parabolic p-Laplacian equation, once again with linear source-type noise as in \eqref{Eq_mult_noise_SPME}. The great step forward for this technique, however, is the amount of flexibility it has. For instance, a modified version of it was recently employed in \cite{GK_fsop3} to prove FSOP for solutions to a class of  stochastic thin-film equations with nonlinear conservative noise, underlining that the basic idea of \cite{Fischer_Grun_FSOP} is restricted neither to second-order equations nor to equations with linear source-type noise.

In addition to results on FSOP, the techniques of \cite{Fischer_Grun_FSOP} also enable one to derive a sufficient condition on $u_0$ which implies the existence of a WTP. Similar to the deterministic case, the sufficient condition is an integral decay estimate on the initial data of the form
\[
    \int_{B_{R_0+r}\setminus B_{R_0}} u_0^{\alpha+1} \, \leq \, g(r),
\]
where $\alpha\geq 0$, $g$ is a function which decays to zero as $r$ tends to zero, and $B_{R_0}$ takes the role of $B'$ in the previous discussion. In the deterministic case, such integral conditions are known to be optimal with the correct choice of $g$ (cf. \cite{ChipotSideris1985,Giacomelli_Grun_WTP_general} and references therein). In \cite{Fischer_Grun_FSOP}, the sufficient decay condition obtained coincides with the deterministic one up to a logarithmic perturbation. Until this work, these results, along with similar findings on the p-Laplacian in \cite{GrillmeierDissertation}, were the only results on waiting time phenomena for stochastic degenerate parabolic equations.

\subsection{Discussion of the results}\label{SS_discussion}
Our results are not only the first regarding the support propagation of \eqref{Eq_SPME_intro}, but more generally the first to establish FSOP for  SPDEs with  conservative Stratonovich noise, and our sufficient condition for the occurrence of a WTP is the first for SPDEs with any type of conservative noise. Our assumptions on the initial value $u_0$ and the noise $W$ -- cf.\ \eqref{Eq_ass_u0} and  \eqref{Eq_ass_B1}--\eqref{Eq_ass_B3} in Section \ref{Sec_rig} -- are the same as required in \cite{Fehrman_Gess_ARMA} for the well-posedness of the equation \eqref{Eq_SPME_intro}.
In particular, our integrability assumption on $u_0$ is weaker than the one required in
\cite{Fischer_Grun_FSOP} as it is not directly coupled to the parameters $(m,n)$, and no limitations on the spatial dimension are imposed. Here, we benefit from the results of \cite{Fehrman_Gess_ARMA } which provide this flexibility in the underlying existence result.
 
Regarding our assumptions on the coefficients, the necessity to assume \eqref{Eq_m} is obvious because the heat equation has infinite propagation speed. In particular, we do not expect the convective term of Stratonovich type to counteract this property of the heat operator since, in the spirit of the Wong--Zakai theorem, one can think of solutions to \eqref{Eq_SPME_intro} as the limit of $u_\epsilon$ satisfying 
\begin{equation}\label{eqn2}
\begin{cases}
	\partial_t  u_\epsilon \,=\, \Delta (u_\epsilon^m) \,-\, \diver( u_\epsilon^n W_\epsilon (\omega)),\\
	u_\epsilon(0) \,=\, u_0,\end{cases}
\end{equation}
where $W_\epsilon$ are suitable regularizations of the driving noise $W$. 

Concerning \eqref{Eq_n}, the upper bound is again in line with the assumptions of the well-posedness result \cite{Fehrman_Gess_ARMA}, and can be explained by the scaling criticality of the endpoint case $n=(m+1)/2$. We also note that this case is interesting in its own right, since we recover precisely the sufficient condition of \cite{Fischer_Grun_FSOP} for the case of linear source-type noise. 

Our need to exclude the linear noise case $n=1$ is of a both technical and foundational nature (see Subsection \ref{SS_proof} for the technical issues). Guided by our intuition that, for the linear noise case $n=1$, we can also think of the solution to \eqref{Eq_SPME_intro} as limit of solutions to porous media equations perturbed by a deterministic convective term as in \eqref{eqn2}, one  expects an instantaneous forward motion of the support to be possible, no matter how flat the initial profile is, i.e., there should be no WTP. Another reason to believe that this is the case can be found in Theorem \ref{Thm_fsop}. Namely, our sufficient condition for a WTP is governed by the growth of the noise term in \eqref{Eq_SPME_intro}, in contrast to the previously discussed results of \cite{Fischer_Grun_FSOP,GrillmeierDissertation} where source-type noise is considered. If one allows $n\searrow1$, the function controlling the growth condition converges to zero, which heuristically suggests that the only initial data for which one may expect a WTP when $n=1$ is $u_0=0$.

Just as is the case in the deterministic setting  (s. \cite{Giacomelli_Grun_WTP_general}), we expect that the techniques employed in this work for the existence of WTPs will find broad applicability to other degenerate parabolic SPDEs such as the stochastic parabolic p-Laplacian with other choices of noise, systems of degenerate parabolic SPDEs, and degenerate parabolic SPDEs of higher order. 

\subsection{Applications}\label{SS_applications}
Next to its mathematical interest, the notoriety of
 \eqref{Eq_SPME_intro} in recent years is due to its role in applications, cf.\ \cite[Section 1.1]{FG_pathwise}: Consider for instance a zero range process ${\eta}^N(t,x)$ on the periodic lattice $\{0,\dots, N - 1\}^d$ with local jump rate $dk^m$ and sufficiently fast decaying particle size $\chi_N\to 0$. Then, the macroscopic profile $\chi_N {\eta}^N(\sqrt{2}N^2 \chi_N^{m-1} t, N x)$ converges to the solution of the porous media equation $\partial_t u = \Delta(u^m)$. Moreover, as predicted in \cite{zimmer} and recently made rigorous in \cite{FG_inventiones,Gess_Heydecker}, the large deviation rate functional of the fluctuations around this profile  coincides with the one of \eqref{Eq_SPME_intro} for $\sigma(u)=u^{m/2}$ for suitably rescaled driving noises (note that our results take effect in this case for $m>2$). This renders the above SPDE an effective mesoscopic model for these underlying particle dynamics. 
 
If we instead begin with a system of mean-field SDEs of the form 
\begin{align}
    \d X^i = \sum_{k=1}^K a_N^k\Bigl(X^i, \frac{1}{N} \sum_{j\ne i} \delta_{X^j}\Bigr) \circ \d \beta_k + b_N\Bigl(\frac{1}{N} \sum_{j\ne i} \delta_{X^j}\Bigr) \d B_i,\qquad i=0,\dots, N,
\end{align}
for independent Brownian motions $(\beta_k)_{k=1}^K$ and $(B_{i})_{i=1}^N$, then, following the theory of mean field dynamics \cite{lasry_lions,carmona_book}, the conditional density of the empirical law of $X$ with respect to $(\beta^k)_{k=1}^K$ converges informally as $N\to\infty$ to the solution of \eqref{Eq_SPME_intro}, provided $a^k_N$ and $b_N$ converge appropriately.

\subsection{Strategy of the proof}\label{SS_proof}
Let us outline how we achieve our results:
Starting with the proof of FSOP,  we follow  the route of \cite{Fischer_Grun_FSOP} and apply an iteration technique based on moment estimates for spatially localized energies. Our starting point is the estimate 
\begin{align}\begin{split}\label{Eq_energy_ineq_intro}
		&
		\E\biggl[\biggl(
		\sup_{t\le \tau} \int_{\T^d} \zeta^2 u^{\alpha+1 }\,\d x \,+\, \int_0^\tau \int_{\T^d}   \zeta^2 | \nabla (u^{(\alpha+m )/2 })|^2\,\d x\, \d t \biggr)^p
		\biggr]^{1/p}
		\\&\quad \lesssim \,\int_{\T^d} \zeta^2 u_0^{\alpha+1 }\,\d x \,+\, \E \biggl[\biggl(\int_0^\tau \int_{\T^d} u^{\alpha+m}|\nabla \zeta|^2 
		\, +\, 
		u^{\alpha+2n-1} \bigl(| \zeta \Delta  \zeta| + |\nabla \zeta|^2 + \zeta^2 \bigr)
		\,\d x
		\,\d t \biggr)^p
		\biggr]^{1/p},
	\end{split}
\end{align}
for any $\zeta\in C^2(\T^d)$, $p\in [1,\infty)$, {$\alpha\in [1,\infty)$} and stopping time  $\tau$, which we prove to be satisfied by kinetic solutions $u$ to \eqref{Eq_SPME_intro} under the assumptions \eqref{Eq_m}, \eqref{Eq_n}, and for sufficiently integrable initial data $u_0\colon \T^d \to \R_{\ge 0}$.

In our setting, the goal for proving FSOP is to show that for a given ball $B_r(x_0)$, whose closure is disjoint from the initial support, there exists a $\P$-a.s. positive stopping time $\tau(\omega)>0$ such that
\begin{equation}
\label{GG-502}
\sup_{t\leq \tau(\omega)}\, \int_{B_r(x_0)} u(\omega,s,x) \,\d x = 0.
\end{equation}
To prove the existence of a WTP, we wish to prove the same when $\partial B_r(x_0)\,\cap \, \mathrm{supp}(u_0)\neq\emptyset$. Similar to the crucial integral estimate (3.49) in \cite{GK_fsop3}, the estimate \eqref{Eq_energy_ineq_intro} involves multiple power-law nonlinearities on the right-hand side due to the generally different scaling behavior of the porous media and noise term in \eqref{Eq_SPME_intro}. For this reason, to achieve an identity in the spirit of \eqref{GG-502} in \cite{GK_fsop3}, accumulated indicator functions have been used -- given as a subtle, rather involved time-weighted sum of the power-law nonlinearities on the right-hand side.  

In order to circumvent such technical difficulties, we introduce a new iteration method which yields the estimate
\begin{equation}\label{Eq7}
        \e{ \sup_{(0,\tau)} \int_{B_{r}} u^{\alpha+1} } \lesssim \int_{B_{r+\delta}} u_{0}^{\alpha+1}+ \sum_{i=1}^3\e{\frac{\tau}{\delta^{\ell_i}} \left( \sup_{(0,\tau)}\int_{B_{r+\delta}} u^{\alpha+1}\right)^{\kappa_i} }.
\end{equation}
with numbers $\kappa_i>1$ and $\ell_i\geq 0$, $i\in \{1,2,3\}.$ 

To establish \eqref{Eq7}, we use a cut-off function $\zeta_\delta\in C^\infty_c(B_{r+\delta}(x_0))$ which coincides with $\mathbf{1}_{B_r(x_0)} $ on $B_r(x_0)$ in \eqref{Eq_energy_ineq_intro}.
We exploit the fact that all the space-time integrals on the right-hand side of \eqref{Eq_energy_ineq_intro} may be estimated against the terms on the left-hand side by combining Gagliardo--Nirenberg and Young's inequalities. Then, we benefit
directly from our new iteration trick: Lemma~\ref{lemma-iteration-trick}  from Appendix~\ref{App_B} can be applied to derive \eqref{Eq7}. Inequality~\eqref{Eq7} has a lot in common with the classical Stampacchia lemma  \cite{Stampacchia_Lemma_1963} (or appropriate variants for the case of multiple summands,  cf. \cite{DalPasso_Giacomelli_Grun_PISA, Grun2001b, GrunPoincare, GS,Utley_TFPME}). As in traditional cases, the fact that $\kappa_i>1$ for $i\in \{1,2,3\}$ is of critical importance. It shall be guaranteed here by our assumptions \eqref{Eq_m} and \eqref{Eq_n}, but it is precisely this technical aspect that fails if one takes $n=1$. 

In contrast to the classical case, however, the argument is concluded not by a simple induction argument, but by the aforementioned filtering technique developed in \cite{Fischer_Grun_FSOP}, which uses \eqref{Eq7} to derive bounds on the probability of the event
\begin{equation}\label{gleichung}
    \left\{ \sup_{(0,t)}\int_{B_r} u >0 \right\} \, \subset \, \bigcup_{k\in \N} \left\{ \sup_{(0,t)}\int_{B_{r_k}} u^{\alpha+1} >\mu_k,  \quad \sup_{(0,t)}\int_{B_{r_{k-1}}} u^{\alpha+1} \leq \mu_{k-1}  \right\}
\end{equation}
for arbitrary $t>0$. Here the sequence $(r_k)_{k\in \N_0}$ is such that $r_k\searrow r$, and $(\mu_k)_{k\in \N_0}$ is such that $\mu_k\searrow 0$.

The occurrence of WTP can be verified using similar arguments. There however, since the balls on the right-hand side of \eqref{gleichung} intersect with the support of $u_0$, 
it is indispensable to take the growth of the initial data in a neighborhood of the boundary into account. 
In order to single out the sufficient growth condition, we cannot directly apply arguments like in \cite{Fischer_Grun_FSOP}, as  multiple power-law nonlinearities with exponents $\kappa_i$ are present on the right-hand side of \eqref{Eq7}. 
By suitably adapting the argument,  we identify the growth conditions arising from the individual non-nonlinearities, of which we then choose the strictest to determine our final condition.

Finally, regarding the derivation of  \eqref{Eq_energy_ineq_intro},  
we note  that  the kinetic formulation is a way of  incorporating large classes of a-priori estimates in the notion of solutions -- cf.\ \cite{LPT_94, Deb_Vovelle_scalar} where kinetic solutions to deterministic and stochastic conservation laws are characterized as entropy solutions in the sense of Kruzkov \cite{kruzkov_ent}. The formulation \eqref{Eq_kin_formulaiton} below, however, leads directly to an  It\^o expansion of $\int_{\T^d} \zeta^2 \Phi( u)  \d x$ 
 only for functions $\Phi$ which are constant near $0$ and $\infty$. Formula \eqref{Eq7} indicates already that our analysis requires power-law functions like 
$\Phi(s)=s^{\alpha+1}$. For this reason, we revisit the viscous regularization procedure from \cite[Section 5]{Fehrman_Gess_ARMA} to  show a variant of \eqref{Eq_energy_ineq_intro} on the approximate level and subsequently prove convergence of these estimates along the passage to the limit.
Similarly as in the case of inequality (4.12) in \cite{Fischer_Grun_FSOP} and of inequality (3.49) in \cite{GK_fsop3}, the derivation of such an estimate is technically demanding. Note that the  additional viscous term destroys the finite speed of propagation property for the approximate solutions and the corresponding  terms  need to be eliminated in the limit on both sides of the inequality. We achieve the latter  by leveraging  the global version of \eqref{Eq_energy_ineq_intro}, i.e., the above energy estimate with $\zeta =\mathbf{1}_{\T^d}$.

\subsection{Organization of the manuscript}
The outline of this article is  as follows: In  Section~ \ref{Sec_rig}, we recall some definitions and results of \cite{Fehrman_Gess_ARMA} on the unique existence of kinetic solutions to \eqref{Eq_SPME_intro}. Then, we   formulate the assumptions and we present the main results of the paper. In Section~\ref{Sec_not}, we collect some notation and an interpolation estimate which will be used throughout the manuscript, and  in  Section~\ref{Sec_loc_est}, we prove the localized energy dissipation inequality \eqref{Eq_energy_ineq_intro} -- up to  a technical auxiliary result  which is presented in Appendix~\ref{Appendix_Ito}.
With the energy estimate in hand, we present in Section~\ref{Sec_FSOP} the stochastic  iteration procedure which  entails

\begin{enumerate}[label=(\roman*)]
\item the result on finite speed of propagation for solutions to \eqref{Eq_SPME_intro},
\item sufficient conditions on initial data for the occurrence of a waiting time phenomenon. 
\end{enumerate}
Iteration arguments and algebraic manipulations which are of a rather technical nature are contained in the Appendices~ \ref{App_B}--\ref{App_C}.

\section{Preliminaries and rigorous statement of our results}\label{Sec_rig}
\subsection{Assumptions and background on kinetic solutions} 
 We recall the conditions \eqref{Eq_m} 
 and \eqref{Eq_n} on the coefficient functions of \eqref{Eq_SPME_intro}:
 \begin{equation}\label{equation_n_m_together}\tag{A1}
	m\in (1,\infty)    \quad \qquad \text{and}
    \qquad \quad
    \sigma(u) \,=\, u^n, \qquad 
n\in (1,(m+1)/2].
 \end{equation}
 Then, for the well-posedness theory from \cite{Fehrman_Gess_ARMA}, we impose on the initial value that
\begin{align} 
    u_0 \,\in \, L^{\alpha +1}(\T^d), \quad u_0 \ge 0,\qquad &\alpha \ge \max\{1 , 4n - 4- m\}, 
    \tag{A2}\label{Eq_ass_u0}
    \end{align}
    where the technical assumption that $\alpha \ge 4n - 4-m$ is in place to ensure \cite[Assumption 5.2, Item 7]{Fehrman_Gess_ARMA}.
From the noise in \eqref{Eq_SPME_intro} we demand that
    \begin{align}
     W(t,x) \, = \, \sum_{k=1}^\infty \psi_k(x) \beta_k(t) ,\qquad &\psi_k \in C^1(\T),
    \tag{A3}\label{Eq_ass_B1}
\end{align}
where $(\beta^{k})_{k\in \N}$ are independent $\R^d$-valued Brownian motions on a complete probability space $(\Omega, \mathfrak{A}, \P)$  with respect to a complete, right-continuous filtration $\mathscr{F}$, which we fix throughout the manuscript. We require that
\begin{equation}
    \Psi_1 \,=\, \sum_{k=1}^\infty \psi_k^2,\quad 
    \Psi_2 \,=\, \frac{1}{2}\sum_{k=1}^\infty \nabla (\psi_k^2),\quad \Psi_3 \,=\, \sum_{k=1}^\infty |\nabla \psi_k|^2 \qquad \text{are continuous on $\T^d$,}
      \tag{A4}\label{Eq_ass_B2}
\end{equation}
with 
\begin{equation}\diver(\Psi_2)  \quad \text{is bounded on $\T^d$},
\tag{A5}\label{Eq_ass_B3}
\end{equation}
as in  \cite[Assumption 2.1]{Fehrman_Gess_ARMA}. In particular,  the It\^o-formulation of \eqref{Eq_SPME_intro} reads then 
\begin{equation}\label{Eq_sPME_Ito}
	\d u \,=\, \Delta (u^m) \,\d t 
	\,+\,\frac{1}{2} \diver\bigl(
	\Psi_1 (\sigma'(u))^2 \nabla u \,+\, \sigma'(u)\sigma(u) \Psi_2
	\bigr)\, \d t
	\,-\, \sum_{k=1}^\infty \diver(  \sigma (u) \psi_k  \, \d \beta_k),
\end{equation}
cf. \cite[Section 2]{Fehrman_Gess_ARMA}.
To rigorously state our results, we recall the definitions of kinetic measures and kinetic solutions to \eqref{Eq_SPME_intro} from \cite[Definitions 3.1, 3.2]{Fehrman_Gess_ARMA}, which can be obtained by formally computing the evolution of the \emph{kinetic function} $(t,x,v )\mapsto \mathbf{1}_{\{0<v <u(t,x)\}} $ by means of It\^o's formula.  There,  $\delta$ denotes the Dirac mass at $0$ and $\beta_k^{(i)}$ the $i$-th entry of the $\R^d$-valued Brownian motion $\beta_k$. We also fix a finite time horizon $T\in (0,\infty)$ for the remainder of this manuscript.
\begin{defn}[Kinetic measure]\label{defi_kin_measure}
A kinetic measure is a mapping $q$ from  $\Omega$  to the space of non-negative, locally finite measures on  $\T^d \times (0,\infty)\times[0,T] $ such that
\begin{equation}
(\omega,t) \,\mapsto \int_0^t \int_{\R} \int_{\T^d}  \psi (x,v ) \,\d q(\omega)
\end{equation}
defines an $\mathscr{F}$-predictable process, for all $\psi \in C_c^\infty( \T^d\times(0,\infty))$.
\end{defn}
\begin{defn}[Kinetic solution to \eqref{Eq_SPME_intro}]Let $u_0\in L^1(\T^d)$ be non-negative and assume \label{defi_kin_sol} \eqref{equation_n_m_together}  and \eqref{Eq_ass_B1}--\eqref{Eq_ass_B3}. Then, a kinetic solution to \eqref{Eq_SPME_intro} is a non-negative, continuous, $\mathscr{F}$-adapted, $L^1(\T^d)$-valued process $u\in L^1(\Omega \times [0,T]\times \T^d)$, such that:
	\begin{enumerate}[label=(\roman*)]
	\item mass is conserved, i.e., $\P$-a.s, we have  $\|u(t)\|_{L^1(\T^d)} = \|u_0\|_{L^1(\T^d)}$ for all $t\in [0,T]$,
	\item it holds $\sigma(u) \in L^2(\Omega \times [0,T]\times \T^d)$,
	\item  we have $((K\wedge u )\vee 1/K) \in  L^2(\Omega \times [0,T]; H^1(\T^d))$ for all $K\in \N$,
	\end{enumerate}  
and there exists  a kinetic measure $q$ with 
\begin{enumerate}[label=(\roman*)]
	\setcounter{enumi}{3}
	\item $\P$-a.s, $m\delta(v - u)v^{m-1} |\nabla u|^2 \le q $ in the sense of measures on $\T^d \times (0,\infty) \times [0,T]$,
	\item \label{Cond_smallness_q} $q$ vanishes at infinity in the sense that $\lim_{M\to\infty} \E [ q(\T^d \times [M,M+1] \times [0,T])] = 0$,
	\item and for every $\psi \in C_c^\infty(\T^d\times (0,\infty ))$, it holds $\P$-a.s.\ 
	\begin{align}\begin{split}\label{Eq_kin_formulaiton}&\int_{\R}\int_{\T^d}
		\mathbf{1}_{\{0<v <u(t,x)\}} \psi(x,v)
		\,\d x \d v \,=\, \int_{\R}\int_{\T^d}
		\mathbf{1}_{\{0<v <u_0(x)\}} \psi(x,v)
		\,\d x \d v\\&\qquad -\,m\int_0^t \int_{\T^d}
		u^{m-1}\nabla u \cdot (\nabla \psi )(x, u)
		\,
		\d x
		\d s\,-\, \frac{1}{2}\int_0^t \int_{\T^d}\bigl[\Psi_1 (\sigma'(u))^2 \nabla u \,+\, \sigma'(u) \sigma(u)\Psi_2 \bigr] \cdot (\nabla \psi )(x, u) \,\d x \d s
		\\&\qquad -\,\int_0^t  \int_{\R} \int_{\T^d} \partial_v \psi (x,v) \,\d q\,+\,\frac{1}{2}\int_0^t \int_{\T^d} \bigl(
		\sigma(u)\sigma'(u)\nabla u \cdot \Psi_2 \,+\, \Psi_3 \sigma^2(u)\bigr) (\partial_v\psi)(x,u) 
		\,\d x
		\d s
		\\&\qquad - \sum_{i=1}^d\sum_{k=1}^\infty \int_0^t \int_{\T^d} \psi(x, u) \partial_i ( \psi_k \sigma (u)) \,\d x \,\d \beta_k^{(i)}(s)		
		,
		\end{split}
	\end{align}
	for all $t\in [0,T]$.
\end{enumerate}
\end{defn}
Under the above assumptions that $u_0\in L^1(\T^d)$ is non-negative, \eqref{equation_n_m_together}  and \eqref{Eq_ass_B1}--\eqref{Eq_ass_B3}, it follows from \cite[Theorem 4.6]{Fehrman_Gess_ARMA} that {kinetic solutions} to \eqref{Eq_SPME_intro} are unique. The additional integrability assumption from \eqref{Eq_ass_u0} on $u_0$ ensures also existence of kinetic solution by \cite[Theorem 5.25]{Fehrman_Gess_ARMA}, which, by uniqueness is then independent of the $\alpha$ in \eqref{Eq_ass_u0}.

\subsection{Rigorous statement of results}
Before we present the main results, we give rigorous definitions for the phenomena which we are describing. In this work, we will make use of the following quantity:

\begin{defn}[Waiting Times]\label{def:fsopwtp}
    Let $u$ be a non-negative, continuous, and $\mathscr{F}$-adapted,  $L^1(\mathbb T^d)$-valued process defined on $[0,T]$. For a point $x_0\in \mathbb T^d\setminus \mathrm{supp}(u_0)$ we define $R_0(x_0):= \mathrm{dist}_{\T^d}(x_0,\mathrm{supp}(u_0)) $. Then, for any $r\in (0,R_0]$, we say that $B_r(x_0)$ has \emph{waiting time $\tau_r(x_0)$}, where  
    \begin{equation}
        \label{eq_defn_hole-filling}
        \tau_r(x_0) := \inf\left\{ t>0, \quad \int_0^t\int_{B_r(x_0)} u >0 \right\} \, \wedge\, T.
    \end{equation}
\end{defn}
Since $u$ is assumed to be $\mathscr F$-adapted with continuous, $L^1(\T^d)$-valued sample paths, the times $\tau_r(x_0)$ are stopping times.
We remark moreover that in the above and below we omit the differentials $\d x \d s$, whenever it is evident that we integrate against the Lebesgue measure.  
Moreover, we drop the reference to $x_0$ whenever it is apparent from the context.

\begin{defn}[FSOP \& WTP]\label{defn_fsop}Let $u$ be a non-negative, continuous, and $\mathscr{F}$-adapted $L^1(\mathbb T^d)$-valued process $u$ defined on $[0,T]$.
\begin{enumerate}[label=(\roman*)]
    \item We say that $u$ has \emph{finite speed of propagation} (FSOP), if for any $x_0 \in \T^d\setminus \mathrm{supp}(u_0)$  and $r\in (0,R_0(x_0))$, we have 
    \[
    \PP\left( \tau_r(x_0) >0  \right) = 1.
    \]
    \item \label{item_defn_2} For $x_0 \in \T^d\setminus \mathrm{supp}(u_0)$, we say  that $u$ exhibits a \emph{waiting time phenomenon}  (WTP) associated with the ball $B_{R_0}(x_0)$, if $\P(\tau_{R_0} >0) = 1$.
\end{enumerate}
\end{defn}

\begin{remark}\label{rem: forwardWTP}
  As already pointed out in the introduction, Definition~\ref{defn_fsop} \ref{item_defn_2} describes a forward waiting time phenomenon in the sense that the boundary of the solution's spatial support does not precede for  a positive random time while instantaneous receding with positive probability is not excluded.
  \end{remark}

  Before stating the main result,  we define, for parameters $\delta>0$ and collections
 $\kappa = (\kappa_i)_{i\in I}$ such that $\kappa_i>1$ for every $i\in I$, the set
 \begin{equation}\label{GG-137}
        \mathscr C_{\delta,\kappa}:= \left\{ f:\R_+ \ra \ \R_+ \, \bigg\vert \, \exists(a_k)_{k\in \N} \subset \R_+ \, \text{ s.t. } \,  \forall \kappa_i\in \kappa: \,\sum_{k\in \N} a_{k-1}^{\kappa_i}a_k\inv  < +\infty, \, \sum_{k\in \N} f\left(\delta \cdot 2^{-k}\right)a_k^{-1} < +\infty \right\}.
    \end{equation}
\begin{thm}[FSOP \& WTP for \eqref{Eq_SPME_intro}]\label{Thm_fsop}
  Let $u$ be the unique kinetic solution to \eqref{Eq_SPME_intro} under the assumptions \eqref{equation_n_m_together}--\eqref{Eq_ass_B3}. Then $u$ has finite speed of propagation.  Furthermore, if, for a given $x_0 \in \T^d\setminus \mathrm{supp}(u_0)$, there exists $r_0>0$ such that
  for the parameter 
  \begin{equation}
    \label{GG-999}
    \theta =  \frac{2(\alpha+1)+d(m-1)}{2(n-1)}
  \end{equation}
  and some $f\in \mathscr C_{2r_0,\kappa}$, where 
    \[
        \kappa = \left( \frac{\alpha+m}{\alpha+1}, \frac{2(\alpha+2n-1) + d(m+1-2n)}{2(\alpha+1)+d(m+1-2n)} \right), 
    \]
  we have
    \begin{equation}\label{eq_flatness}
        \sup_{r\in (0,r_0)} \ \frac{1}{f(r)r^{\theta}} \ \int_{B_{R_0+r}(x_0)\setminus B_{R_0}(x_0)} u_0^{\alpha+1} < +\infty, 
       \end{equation}
        then $u$ exhibits a waiting time phenomenon  associated with the ball $B_{R_0}(x_0)$. 

\end{thm}

\begin{remark}
    Note that \eqref{equation_n_m_together} implies that $\kappa_i >1$ for $i=1,2$. Moreover, we remind the reader that the radius $R_0$ in the formulation of Theorem~\ref{Thm_fsop} is as defined in Definition~\ref{def:fsopwtp}.
\end{remark}

\begin{remark}
    We note that the condition \eqref{eq_flatness} with the set $\mathscr C_{\delta,\kappa}$ is the sharpest possible assumption using the techniques employed in Section~\ref{Sec_FSOP}. In the deterministic setting, there is no need for the perturbations caused by $f(r)$, i.e., it is enough that \eqref{eq_flatness} is satisfied with $f(r)r^\theta$ replaced with simply $r^\theta$. For this reason it is of note that we may, for any $\delta>0$ and any finite collection $\kappa$, find $f\in\mathscr C_{\delta,\kappa}$ that approach zero as $x\ra 0$ with slower than polynomial speed, hence only perturbing the condition from the deterministic case by a lower-order factor. This was noted in \cite{Fischer_Grun_FSOP}, where powers of  $f(r) \sim \log(r\inv)^{-1}$ are used. Such logarithmic perturbations are reproducible in this formulation, as we show here. Given $\delta>0$ and a finite collection of $\kappa_i>1$, consider 
    \[
        f(x) = \log_2\left(\delta\cdot  x\inv\right)^{-L}, \quad L>0.
    \]
    Then we have
    \[
        f(\delta\cdot 2^{-k}) = k^{-L}
    \]
    and so we take $a_k = k^{-p}$ with $p$ such that
    \[
    p<L.
    \]
    Then we also see that for a collection $\kappa$ with each $\kappa_i >1$
    \[
        \sum_{k\in \N} a_{k-1}^{\kappa_i}a_k\inv  =  \sum_{k\in \N} \frac{k^p}{(k-1)^{\kappa_i p}}.
    \]
    This series converges if and only if 
    \[
        p > \frac{1}{\kappa_i-1} 
    \]
    for every $i=1,\dots, N$. Hence, if $L$ is large enough, i.e.,
    \[
        L> \max\left\{ \frac{1}{\kappa_i -1}\right\},
    \]
    then $f\in \mathscr C_{\delta,\kappa}$. 
    
\end{remark}

\begin{remark}
    The condition \eqref{eq_flatness} is referred to as a flatness condition since it allows one to compute growth exponents for the initial data near points of touchdown: for simplicity, let $d=1$ and consider initial data which (in a neighborhood of zero) takes the form: 
    \[
        u_0 = \begin{cases}
            x^\gamma & x>0 \\
            0 & x\leq 0
        \end{cases}.
    \]
    Setting $x_0 = -\epsilon $ for some $\epsilon>0$, the condition \eqref{eq_flatness} is satisfied if and only if a neighborhood of $x_0$ exists in which
    \[
    \frac{1}{f(r)r^\theta} \ \int_{0}^{r} x^{\gamma(\alpha+1)} 
    \]
    is bounded. Simply computing the integral, one requires boundedness of 
    \[
        f(r)\inv \cdot r^{\gamma(\alpha+1)+1-\theta}
    \]
    as $r\ra 0$. Since we have found an example of $f$ such that $f\inv$ blows up slower than $r\inv$, we obtain boundedness whenever $\gamma$ is large enough, imposing a polynomial-type growth condition on the initial data. 
\end{remark}
By estimating similarly to the preceding remark 
\[
\int_{B_{R_0+r}(x_0)\setminus B_{R_0}(x_0)} u_0^{\alpha+1} \,\lesssim\, r^d\left(\sup_{B_{R_0+r}(x_0)\setminus B_{R_0}(x_0)} u_0 \right)^{\alpha+1}
\]
in any spatial dimension, we obtain the following  corollary:
\begin{col}
    Let $u$ be the unique kinetic solution to \eqref{Eq_SPME_intro} under the assumptions \eqref{equation_n_m_together}--\eqref{Eq_ass_B3} and let $x_0 \in \mathbb T^d\setminus \mathrm{supp} ( u_0)$. Suppose that there exists $r_0 >0$ such that 
    \begin{equation}\label{eq_stronger_flatness}
        \sup_{r\in (0,r_0)} \, \sup_{B_{R_0+r}\setminus B_{R_0}} \frac{u_0(x)}{{f(r)^{\frac{1}{\alpha+1}}r^\gamma}} < +\infty, \quad \gamma =  \frac{\theta-d}{\alpha+1} = 
        \frac{2(\alpha+1) + d(m+1-2n)}{2(n-1)(\alpha+1)},
    \end{equation}
    where $\theta$ is the constant from \eqref{eq_flatness}.
    Then $u$ has a WTP associated with the ball $B_{R_0}(x_0)$.
\end{col}
\noindent Note in the above result that in the case $n=(m+1)/2$, where the noise term scales like the porous media term, we recover the known optimal condition for deterministic porous media equations (multiplied by a lower-order factor).

\section{Notation and Auxiliary Results}\label{Sec_not}
Throughout this manuscript, 
$(\Omega, \mathfrak{A}, \P)$ is a fixed, complete probability space with a right-continuous and complete filtration $\mathscr{F}$. The family $(\beta_k)_{k\in \N}$ consists of independent, $\R^d$-valued $\mathscr{F}$-Brownian motions and we write $\E$ for the evaluation of the expectation with respect to $\P$.
From here on we impose the assumptions \eqref{equation_n_m_together}--\eqref{Eq_ass_B3} for a family $(\psi_k)_{k\in \N}$ of $C^1(\T^d)$-functions without further mentioning it. In particular, we have $\sigma(u) = u^n$, the parameter $\alpha$ is chosen such that \eqref{Eq_ass_u0} holds true and the formulas \eqref{Eq_ass_B1}--\eqref{Eq_ass_B2} give rise to the Wiener process $W$ as well as the functions $\Psi_1$, $\Psi_2$ and $\Psi_3$. We remark that when denoting integration against the Lebesgue measure we sometimes omit  the differentials $\d x $ and/or $\d t$ to shorten our formulas. In that context we also shorten the domain of integration by writing $\Gamma_\tau$ for the random set $\Gamma \times [0,\tau]$, where $\tau$ is a stopping time.

We also record, where the following recurring objects are introduced:
\begin{itemize}
    \item Kinetic solutions to \eqref{Eq_SPME_intro}, typically denoted by $u$, are defined in Definitions \ref{defi_kin_measure}--\ref{defi_kin_sol};
    \item The regularized noise coefficient $\sigma_l$ is chosen  in \eqref{Eq_21} based on a suitable cutoff function $\phi$;
    \item Weak solutions to \eqref{Eq_viscous_reg}, typically denoted by $u_{\kappa,l}$, are defined in Definition \ref{def_weak_sol};
    \item The metric $d_{w,b}$ for weak convergence in bounded subsets of $L^2([0,T]\times \T^d;\R^d)$ is defined in \eqref{Eq_d_wb};
    \item The functions $\zeta_{s,\delta}$ are introduced in \ref{L1}--\ref{L4};
    \item The parameters $(\kappa_i,\ell_i)$ are defined in (the proof of) Lemma \ref{lemma-iteration-expectation};
    \item The process $G$ is defined in \eqref{equation_for_G}.
\end{itemize}
For completeness, we state the following version of Gagliardo--Nirenberg's inequality -- for a proof, we refer to \cite{Passo_Giacomelli_Shishkov}.
 \begin{lemma} \label{app-L-1}
Let $ 1 \le r \le \infty$, $0< q <p$, $ m \in \N_+$
such that
\[
\frac 1r - \frac mN < \frac 1p \,.
\]
If $\mathcal{O} \subset \R^N$ is bounded with piecewise smooth boundary,
then positive constants $c_1$ and $ c_2$ depending only on $\mathcal{O},
r, p, m$ and $q$ exist such that for any $u \in L^q (\mathcal{O})$
satisfying $ D^m u \in L^r (\mathcal{O})$, the following inequality holds:
\begin{equation}
\| u \|_p \le c_1 \|D^m u\|^a_r \| u \|^{1-a}_q + c_2 \| u \|_q
\label{eq-GN}
\end{equation}
where $a = \frac{\frac 1q - \frac 1p}{\frac 1q + \frac m N - \frac
1r}$.

Especially, if $\mathcal{O}=\R^N$,  then (\ref{eq-GN}) holds
with constant $c_1=c(r,p,m,q)$ and $c_2=0$.
\end{lemma}

\section{Localized energy estimates for kinetic solutions}\label{Sec_loc_est}

A main ingredient to prove Theorem \ref{Thm_fsop} is a localized energy estimate for \eqref{Eq_SPME_intro} for which we recall that the proof of existence of kinetic solutions from \cite{Fehrman_Gess_ARMA} relies on the regularization 
\begin{equation}\label{Eq_viscous_reg}\tag{sPME-reg}\begin{cases}
  \d u \,=\, \Delta (u^m) \,\d t \,+\,\kappa  \Delta u\,\d t 
    \,+\,\frac{1}{2} \diver\bigl(
    \Psi_1 (\sigma_l'(u))^2 \nabla u \,+\, \sigma_l'(u)\sigma_l(u) \Psi_2
    \bigr)\, \d t
    \,-\, \sum_{k=1}^\infty \diver(  \sigma_l (u) \psi_k   \, \d \beta_k),\\
    u(0) \,=\, u_0
    \end{cases}
\end{equation}
of \eqref{Eq_sPME_Ito},
where 
$\sigma_l \to \sigma$ in $C^1_{\loc}((0,\infty))$ satisfies 
\cite[Assumption 5.2]{Fehrman_Gess_ARMA}
uniformly in $l$ as well as 
\begin{equation}\label{Eq_ass_sigma_l}
    \sigma_l \in C([0,\infty)) \cap C^\infty((0,\infty)) ,\qquad \sigma_l(0) \, =\,  0,\qquad \sigma_l' \in C_c^\infty([0,\infty)),
\end{equation}
cf. \cite[Assumption 5.5]{Fehrman_Gess_ARMA}. 
While $\sigma_l$ satisfying \eqref{Eq_ass_sigma_l} can be generically constructed as in \cite[Lemma 5.18]{Fehrman_Gess_ARMA}, we make throughout this manuscript the following explicit choice: We fix some cutoff function $\phi$, i.e., we let $\phi\in C_c^\infty([0,\infty))$ be decreasing  such that $\phi\equiv
    1$ on $[0,1]$ and $\phi \equiv 0$ on $[2 ,\infty)$ and set \begin{equation}\label{Eq_21}
    \sigma_l(r)\,=\, \bigl((r+1/l)^n  \,-\, (1/l)^n \bigr)\phi(r/l),
    \qquad
    r\in [0,\infty),
    \end{equation}
    for $l\in \N$.
Using that then 
\[
\sigma_l'(r) \,=\, n(r+1/l)^{n-1}\phi(r/l) \,+\, \bigl(
(r+1/l)^n - (1/l)^n
\bigr)\phi'(r/l)/l,
\]
we obtain
\begin{align}\begin{split}\label{Eq_bounds_sigma_l}
0\,\le \, \sigma_l(r)\,&\lesssim_n \,( r^{n-1} \,+\, (1/l)^{n-1})r,
\\ 
|\sigma_l'(r)|
\,&\lesssim_{n,\phi} \, r^{n-1} \,+\, (1/l)^{n-1},
\end{split}
\end{align}
 for $r\ge 0$ with which one can check that this is in fact an admissible approximation of $\sigma(u)=u^n$, i.e., that \eqref{Eq_ass_sigma_l} is satisfied.
Weak solutions to \eqref{Eq_viscous_reg} are defined as follows, cf.\ \cite[Defintion 5.6]{Fehrman_Gess_ARMA}.  
\begin{defn}[Weak solutions to \eqref{Eq_viscous_reg}]\label{def_weak_sol}
    A 
    weak solution to \eqref{Eq_viscous_reg} is a continuous $L^{\alpha+1}(\T^d)$-valued adapted non-negative process $u$ satisfying $u, u^{(\alpha+m)/2}\in L^2([0,T]; H^1(\T^d))$ such that for any $\varphi\in C^\infty(\T^d)$, a.s. for every $t\in [0,T]$:
    \begin{align*}
        \int_{\T^d} u_t \vp \,\d x \,=\,& \int_{\T^d} u_0 \vp \,\d x \,-\, m \int_0^t \int_{\T^d} u^{m-1} \nabla u \cdot \nabla \vp \,\d x\d s
        \,-\,\kappa \int_0^t \int_{\T^d} 
        \nabla u \cdot  \nabla \vp  \,\d x\, \d s
        \\& -\,\frac{1}{2}\int_0^t \int_{\T^d}
        \Psi_1(\sigma_l'(u))^2 \nabla u \cdot \nabla \vp \,+\,
        \sigma_l\sigma_l'(u)\Psi_2 \cdot\nabla\vp
        \, \d x\,\d s\\&+\, \sum_{k=1}^\infty\int_0^t \int_{\T^d}
        \sigma_l(u)\psi_k \nabla \vp \, \d x \cdot  \,
        \d \beta_k(s).
    \end{align*}
\end{defn}
Due to the additional Laplacian and the more regular noise term,  weak solutions $u_{\kappa,l}$ to \eqref{Eq_viscous_reg} can be constructed using a Galerkin scheme under the above assumptions as shown in \cite[Proposition 5.17]{Fehrman_Gess_ARMA}. 
It is the content of \cite[Section 5.3]{Fehrman_Gess_ARMA} that then,
as  $\kappa\to 0$ and $l\to\infty$, these weak solutions converge to the unique kinetic solution $u$  { to \eqref{Eq_SPME_intro}}. 
More precisely, inspecting the proof of \cite[Theorem 5.25]{Fehrman_Gess_ARMA} shows that we have
\begin{align}\begin{split}\label{Eq_convergences}
    u_{\kappa,l} \,\to \, u ,\qquad &\text{in }L^1([0,T] ; L^1(\T^d)), \\
    \nabla (    u_{\kappa,l}^{{(\alpha+m)}/2} ) \,\to\,
    \nabla( u^{{(\alpha+m)}/2}  )
    ,\qquad &\text{in }(L^2([0,T]\times\T^d; \R^d) , d_{w.b}), \\
    \kappa \nabla     u_{\kappa,l}  \,\rightharpoonup \,
    0
    ,\qquad &\text{in }L^2([0,T] ; L^2(\T^d)) .
\end{split}\end{align}
in probability for any sequence $\kappa\to 0$ and $l\to \infty$, where the appearing metric $d_{w,b}$ is the metric given by \begin{equation}\label{Eq_d_wb}
d_{w,b}(u,v) \,=\, \sum_{k=1}^\infty 2^{-k} |\langle u - v,w_k\rangle |
\end{equation}
for $(w_k)_{k\in \N}$ lying dense in the unit ball of $L^2([0,T]\times\T^d; \R^d) $. We remark that the latter metrizes weak convergence on bounded subsets of $L^2([0,T]\times\T^d) $, i.e., the convergence $u_n\rightharpoonup u$ is equivalent to boundedness of $\|u_n\|_{L^2([0,T]\times\T^d) }$ and convergence with respect to $d_{w.b}$. 

{ Due to a subtle technicality involving the application of It\^o's formula (s. Appendix  \ref{Appendix_Ito}), we impose for the moment the purely qualitative assumption that 
\begin{equation}\label{assumption_m_plus_one} \tag{qual}
    u_0 \in L^{m+1}(\T^d),\qquad \text{in case $\alpha<m$.}
\end{equation}
While for the moment in place, we generalize the central results of this section to initial values which do not necessarily satisfy \eqref{assumption_m_plus_one} in Proposition \ref{prop_generalization} using the $L^1$-contraction estimate for kinetic solutions to \eqref{Eq_SPME_intro} established in \cite[Theorem 4.6]{Fehrman_Gess_ARMA}.
}

\begin{lemma}\label{Lemma_localized_approx}Let $\zeta\in C^2(\T^d)$, $p\in [1,\infty)$ and $\tau \le T$ a stopping time. Then it holds
\begin{align}\begin{split}
    \label{Eq_localized_energy_est_approx}
&
 \E\biggl[\biggl(
        \sup_{t\le \tau} \int_{\T^d} \zeta^2 u_{\kappa,l}^{\alpha+1 }\,\d x \,+\, \int_0^\tau \int_{\T^d}   \zeta^2 u_{\kappa,l}^{\alpha+m-2}|\nabla u_{\kappa,l}|^2\,\d x\, \d t \,+\,\kappa \int_0^\tau \int_{\T^d}\zeta^2 u_{\kappa,l}^{\alpha-1}|\nabla u_{\kappa,l}|^2 \,\d x\, \d t\biggr)^p
        \biggr]^{1/p}
        \\&\quad \lesssim_{(\alpha,m,n,p,\phi,\psi)} \,\int_{\T^d} \zeta^2 u_0^{\alpha+1 }\,\d x \,+\, \E \biggl[\biggl(\int_0^\tau \int_{\T^d} u_{\kappa,l}^{\alpha+m}|\nabla \zeta|^2 \,\d x\,\d t \,+\,
        \kappa\int_0^\tau  \int_{\T^d}u_{\kappa,l}^{\alpha+1}|\nabla \zeta|^2\,\d x\, \d t\biggr)^p
        \biggr]^{1/p}
        \\&\qquad +\, \E\biggl[\biggl(
        \int_0^\tau 
        \int_{\T^d} 
        u_{\kappa,l}^{\alpha+2n-1} \bigl(|\zeta \Delta  \zeta| + |\nabla \zeta|^2 +\zeta^2 \bigr)
        \,\d x
        \,\d t \biggr)^p
        \biggr]^{1/p}
        \\&\qquad +\, \frac{1}{l^{2n-2}}\E\biggl[\biggl(
        \int_0^\tau 
        \int_{\T^d} 
        u_{\kappa,l}^{\alpha+1} \bigl(|\zeta \Delta  \zeta | + |\nabla \zeta|^2 +\zeta^2  \bigr)
        \,\d x
        \,\d t \biggr)^p
        \biggr]^{1/p}
\end{split}
\end{align}
for any weak solution $   u_{\kappa,l} $ to \eqref{Eq_viscous_reg} in the sense of Definition \ref{def_weak_sol} with initial value $u_0$ {satisfying \eqref{assumption_m_plus_one}}.
\end{lemma}
\begin{proof} We use It\^o's formula  to compute that
    \begin{align}\begin{split} \label{Eq_Ito}&
         \frac{1}{\alpha+1}\d\int_{\T^d} \zeta^2 u_{\kappa,l}^{\alpha+1} \,\d x \\&\quad=\,- \alpha m \int_{\T^d}\zeta^2 u^{\alpha+m-2}|\nabla u_{\kappa,l}|^2\,\d x\, \d t \,-\,m\int_{\T^d} u^{\alpha+m-1}\nabla (\zeta^2)\cdot \nabla u_{\kappa,l}\,\d x\,\d t
        \\&\qquad
        -\alpha\kappa \int_{\T^d}\zeta^2 u_{\kappa,l}^{\alpha-1}|\nabla u_{\kappa,l}|^2 \,\d x\, \d t\,-\, \kappa \int_{\T^d}u^{\alpha}\nabla (\zeta^2)\cdot \nabla u_{\kappa,l}\,\d x\, \d t
        \\&\qquad
        -\frac{\alpha}{2}\int_{\T^d}\zeta^2 u_{\kappa,l}^{\alpha-1}|\nabla u_{\kappa,l}|^2 \Psi_1 (\sigma_l')^2(u_{\kappa,l}) \,\d x\,\d t
        \,
        -\, \frac{1}{2}\int_{\T^d} u_{\kappa,l}^{\alpha}\nabla (\zeta^2)\cdot \nabla u_{\kappa,l} \Psi_1 (\sigma_l')^2(u_{\kappa,l}) \,\d x\,\d t
        \\&\qquad
        -\frac{\alpha}{2}\int_{\T^d}\zeta^2 u_{\kappa,l}^{\alpha-1}\nabla u_{\kappa,l} \cdot \Psi_2 (\sigma_l' \sigma_l)(u_{\kappa,l}) \,\d x\,\d t
        \,
         -\frac{1}{2}\int_{\T^d} u_{\kappa,l}^{\alpha} \nabla(\zeta^2)  \cdot \Psi_2 (\sigma_l' \sigma_l)(u_{\kappa,l}) \,\d x\,\d t
        \\&\qquad+\, \frac{\alpha}{2}\int_{\T^d}\zeta^2 u_{\kappa,l}^{\alpha-1}(\sigma_l')^2(u_{\kappa,l}) |\nabla u_{\kappa,l}|^2 \Psi_1\,\d x\,\d t 
        +\, \frac{\alpha}{2}\int_{\T^d}\zeta^2 u_{\kappa,l}^{\alpha-1}\sigma_l^2(u_{\kappa,l}) \Psi_3\,\d x\,\d t 
          \\&\qquad
          +\, {\alpha}\int_{\T^d}\zeta^2 u_{\kappa,l}^{\alpha-1}(\sigma_l' \sigma_l)(u_{\kappa,l}) \nabla u_{\kappa,l} \cdot \Psi_2\,\d x\,\d t 
        \\&\qquad
        -\,\sum_{k=1}^\infty \int_{\T^d}\zeta^2 u_{\kappa,l}^\alpha \sigma_l'(u_{\kappa,l}) \psi_k \nabla u_{\kappa,l} \,\d x \cdot \d\beta_k \,-\, \sum_{k=1}^\infty \int_{\T^d}\zeta^2 u_{\kappa,l}^\alpha \sigma_l(u_{\kappa,l}) \nabla \psi_k \,\d x \cdot \d\beta_k
        \\&\quad =\, A_1+\,\dots \,+ A_4 \,+\, B_1+\,\dots \,+ B_7
        \,+\, C_1+ C_2.
        \end{split}\end{align}
        Its application  can be justified, e.g., using the It\^o's formula from \cite[Proposition A.1]{DHV_16} together with a suitable regularization of the function $r^{\alpha+1}/(\alpha+1)$, see Appendix \ref{Appendix_Ito}, and we proceed to estimate the terms on the right-hand side.
        
        \emph{Ad $A$.} By Young's inequality one bounds
        \begin{align*}
           A_1+\,\dots \,+ A_4  \,\le\, & -\frac{\alpha m}{2}\int_{\T^d}\zeta^2 u_{\kappa,l}^{\alpha+m-2}|\nabla u_{\kappa,l}|^2\,\d x\, \d t \,+\,\frac{2m}{\alpha}\int_{\T^d} u_{\kappa,l}^{\alpha+m}|\nabla \zeta|^2 \,\d x\,\d t
        \\&
        -\frac{\alpha\kappa}{2} \int_{\T^d}\zeta^2 u_{\kappa,l}^{\alpha-1}|\nabla u_{\kappa,l}|^2 \,\d x\, \d t\,+\, \frac{2\kappa}{\alpha} \int_{\T^d}u_{\kappa,l}^{\alpha+1}|\nabla \zeta|^2\,\d x\, \d t.
    \end{align*}

    \emph{Ad $B$.} We observe that $B_1 = -B_5$ and $B_3 = -B_7{/{2}}$ so that
    \begin{align*}
         B_1+\,\dots \,+ B_7 \,=\, &
         \frac{1}{2}\int_{\T^d} \int_0^{u_{\kappa,l}} r^{\alpha} (\sigma_l')^2(r) \,\d r (\Psi_1 \Delta (\zeta^2 )\,+\, \nabla (\zeta^2)\cdot \Psi_2) \,\d x \,\d t
         \\& -\frac{1}{2}\int_{\T^d} u_{\kappa,l}^{\alpha} \nabla(\zeta^2)  \cdot \Psi_2 (\sigma_l' \sigma_l)(u_{\kappa,l}) \,\d x\,\d t \,+\, \frac{\alpha}{2}\int_{\T^d}\zeta^2 u_{\kappa,l}^{\alpha-1}\sigma_l^2(u_{\kappa,l}) \Psi_3\,\d x\,\d t 
         \\& {-\frac{\alpha}{2} \int_{\T^d} \int_0^{u_{\kappa,l}} r^{\alpha-1}(\sigma_l\sigma_l')(r) \, \d r  (   \Psi_2 \cdot  \nabla (\zeta^2)+ \zeta^2\diver(\Psi_2))\, \d x\, \d t} 
    \end{align*}
    as follows {from} integration by parts in the first {and last terms}.

    \emph{Ad $C$.}
    Also here we integrate by parts in the first expression to obtain 
    \begin{align*}
         C_1+ C_2 \,=\, \sum_{k=1}^\infty \biggl(\int_{\T^d} \int_0^{u_{\kappa,l}} r^\alpha \sigma_l'(r) \,\d r (  \zeta^2\nabla \psi_k  \,+\,  \psi_k \nabla (\zeta^2 ))\,\d x\,-\, \int_{\T^d}\zeta^2 u_{\kappa,l}^\alpha \sigma_l(u_{\kappa,l}) \nabla \psi_k \,\d x  \biggr) \cdot \d\beta_k.
    \end{align*}
    Inserting this in \eqref{Eq_Ito} and taking the supremum until the stopping time
    \begin{equation}
        \label{Eq_stopping_1}
    \tau_j \,=\, \inf\biggl\{
    t\in [0,T] \,\bigg| \, \sup_{s\le t} \|u_{\kappa,l}(s)\|_{L^{\alpha+1}(\T^d)} \,\ge\, j
    \biggr\} \,\wedge \,\tau,
    \end{equation}
    $p$-th moments as well as the expectation yields
    \begin{align}\begin{split}
        \label{Eq_1}&
        \frac{1}{\alpha+1}\E\biggl[
        \biggl(\sup_{t\le \tau_j} \int_{\T^d} \zeta^2 u_{\kappa,l}^{\alpha+1 }\,\d x\biggr)^p\biggr]^{1/p} \\ &+\, \E\biggl[\biggl(\frac{\alpha m}{2}\int_0^{\tau_j} \int_{\T^d}   \zeta^2 u_{\kappa,l}^{\alpha+m-2}|\nabla u_{\kappa,l}|^2\,\d x\, \d t \,+\,\frac{\alpha\kappa}{2}\int_0^{\tau_j} \int_{\T^d}\zeta^2 u_{\kappa,l}^{\alpha-1}|\nabla u_{\kappa,l}|^2 \,\d x\, \d  t\biggr)^{p}
        \biggr]^{1/p}
        \\&\quad \le \,\frac{2}{\alpha+1}\int_{\T^d} \zeta^2 u_0^{\alpha+1 }\,\d x \,+\, \E \biggl[\biggl(\frac{4m}{\alpha}\int_0^{\tau_j} \int_{\T^d} u_{\kappa,l}^{\alpha+m}|\nabla \zeta|^2 \,\d x\,\d t \,+\,
        \frac{4\kappa}{\alpha}\int_0^{\tau_j}  \int_{\T^d}u_{\kappa,l}^{\alpha+1}|\nabla \zeta|^2\,\d x\, \d t\biggr)^p
        \biggr]^{1/p}
        \\& \qquad +\, \E\biggl[\biggl|\int_0^{\tau_j}\int_{\T^d} \int_0^{u_{\kappa,l}} r^{\alpha} (\sigma_l')^2(r) \,\d r (\Psi_1 \Delta (\zeta^2) \,+\, \nabla (\zeta^2)\cdot \Psi_2) \,\d x \,\d t\biggr|^p\biggr]^{1/p}
        \\& \qquad +\, \E\biggl[\biggl|\int_0^{\tau_j} \int_{\T^d} u_{\kappa,l}^{\alpha} \nabla(\zeta^2)  \cdot \Psi_2 (\sigma_l' \sigma_l)(u_{\kappa,l}) \,\d x\,\d t \,+\, \alpha\int_0^{\tau_j} \int_{\T^d}\zeta^2 u_{\kappa,l}^{\alpha-1}\sigma_l^2(u_{\kappa,l}) \Psi_3\,\d x\,\d t\biggr|^p \biggr]^{1/p}
        \\& \qquad {+\, \alpha \E \biggl[ \biggl|\int_{\T^d} \int_0^{u_{\kappa,l}} r^{\alpha-1}(\sigma_l\sigma_l')(r) \, \d r  (   \Psi_2 \cdot  \nabla (\zeta^2)+ \zeta^2\diver(\Psi_2))\, \d x\, \d t\biggr|^p\biggr]^{1/p} }
        \\&\qquad +\,C_p^{1/p}\E\biggl[\biggl(\sum_{k=1}^\infty
        \int_0^{\tau_j} \biggl|\int_{\T^d}\zeta^2 u_{\kappa,l}^\alpha \sigma_l(u_{\kappa,l}) \nabla \psi_k \,-\,  \biggl(\int_0^{u_{\kappa,l}} r^\alpha \sigma_l'(r) \,\d r\biggr) (  \zeta^2\nabla \psi_k  \,+\,  \psi_k \nabla (\zeta^2) )\,\d x  \biggr|^2\d t
        \biggr)^{p/2}
        \biggr]^{1/p}
        \\&\quad = F_1 \,+\,F_2 \,+\,G_1 \,+\,G_2 { \, + \, G_3} \,+\,H,
    \end{split}\end{align}
    where $C_p$ is the constant from the Burkholder--Davis--Gundy inequality.

    \emph{Ad $G$.} We use \eqref{Eq_ass_B2} {and \eqref{Eq_ass_B3}} together with \eqref{Eq_bounds_sigma_l} in order to estimate
    \begin{align*}
        G_1 \,+\, G_2 \, { + \,G_3 \,}  \lesssim_{(\alpha,n,\phi,\psi)}\,&\E\biggl[\biggl(
        \int_0^{\tau_j} 
        \int_{\T^d} 
        u_{\kappa,l}^{\alpha+2n-1} \bigl(|\zeta \Delta  \zeta| + |\nabla \zeta|^2 +\zeta^2 \bigr)
        \,\d x
        \,\d t \biggr)^p
        \biggr]^{1/p}\\&+\, \frac{1}{l^{2n-2}}\E\biggl[
        \biggl(\int_0^{\tau_j} 
        \int_{\T^d} 
        u_{\kappa,l}^{\alpha+1} \bigl(|\zeta \Delta  \zeta| + |\nabla \zeta|^2 +\zeta^2  \bigr)
        \,\d x
        \,\d t \biggr)^p
        \biggr]^{1/p}.
    \end{align*}

    \emph{Ad $H$.} Similarly, we deduce that
    \begin{align*}
        H\,&\lesssim_{(\alpha, n,p, \phi,\psi )}\, \E\biggl[ \biggl(
        \int_0^{\tau_j}
        \biggl(\int_{\T^d} \zeta^2 u_{\kappa,l}^{\alpha+1}  \,\d x \biggr) 
         \biggl(\int_{\T^d} \bigl(u_{\kappa,l}^{\alpha+2n-1} + u_{\kappa,l}^{\alpha+1}/l^{2n-2} \bigr) \bigl(
         |\nabla \zeta|^2 +\zeta^2  
         \bigr)  \,\d x \biggr) 
        \,\d t
        \biggr)^{p/2}\biggr]^{1/p}
        \\&\le \,\varepsilon \E \biggl[\biggl(
        \sup_{t\le {\tau_j}} \int_{\T^d} \zeta^2 u_{\kappa,l}^{\alpha+1 }\,\d x\biggr)^p
        \biggr]^{1/p} \, +\, \frac{1}{4\epsilon}
        \E\biggl[\biggl(
        \int_0^{\tau_j} 
         \int_{\T^d} \bigl(u_{\kappa,l}^{\alpha+2n-1} + u_{\kappa,l}^{\alpha+1}/l^{2n-2} \bigr) \bigl( |\nabla \zeta|^2+
         \zeta^2 
         \bigr)  \,\d x  
        \,\d t\biggr)^p\biggr]^{1/p}
    \end{align*}
    for any $\varepsilon>0$.
    
    Inserting this in \eqref{Eq_1} and choosing $\varepsilon$ sufficiently small yields the bound
    \begin{align}\begin{split}
&
 \E\biggl[\biggl(
        \sup_{t\le \tau_j } \int_{\T^d} \zeta^2 u_{\kappa,l}^{\alpha+1 }\,\d x \,+\, \int_0^{\tau_j} \int_{\T^d}   \zeta^2 u_{\kappa,l}^{\alpha+m-2}|\nabla u_{\kappa,l}|^2\,\d x\, \d t \,+\,\kappa \int_0^{\tau_j} \int_{\T^d}\zeta^2 u_{\kappa,l}^{\alpha-1}|\nabla u_{\kappa,l}|^2 \,\d x\, \d t\biggr)^p
        \biggr]^{1/p}
        \\&\quad \lesssim_{(\alpha,m,n,p,\phi,\psi)} \,\int_{\T^d} \zeta^2 u_0^{\alpha+1 }\,\d x \,+\, \E \biggl[\biggl(\int_0^\tau \int_{\T^d} u_{\kappa,l}^{\alpha+m}|\nabla \zeta|^2 \,\d x\,\d t \,+\,
        \kappa\int_0^\tau  \int_{\T^d}u_{\kappa,l}^{\alpha+1}|\nabla \zeta|^2\,\d x\, \d t\biggr)^p
        \biggr]^{1/p}
        \\&\qquad +\, \E\biggl[\biggl(
        \int_0^\tau 
        \int_{\T^d} 
        u_{\kappa,l}^{\alpha+2n-1} \bigl(|\zeta \Delta  \zeta| + |\nabla \zeta|^2 +\zeta^2 \bigr)
        \,\d x
        \,\d t \biggr)^p
        \biggr]^{1/p}\\&\qquad +\, \frac{1}{l^{2n-2}}\E\biggl[\biggl(
        \int_0^\tau 
        \int_{\T^d} 
        u_{\kappa,l}^{\alpha+1} \bigl(|\zeta \Delta  \zeta | + |\nabla \zeta|^2 +\zeta^2  \bigr)
        \,\d x
        \,\d t \biggr)^p
        \biggr]^{1/p},
\end{split}
\end{align}
since we stopped the process according to \eqref{Eq_stopping_1}. By  $u_{\kappa,l}\in C([0,T]; L^{\alpha+1}(\T^d))$  we have  $\tau_j\nearrow \tau $ as  $j\to\infty$ and thus the desired estimate  \eqref{Eq_localized_energy_est_approx} follows by Fatou's lemma.
\end{proof}
\begin{col}\label{cor_add_conv} Let  $   u_{\kappa,l} $ for $\kappa\in (0,\infty)$ and $l\in \N$ be weak solutions to \eqref{Eq_viscous_reg} in the sense of Definition \ref{def_weak_sol} with initial value $u_0$ and $u$ be the unique kinetic solution to \eqref{Eq_SPME_intro}  with initial value $u_0$  {satisfying \eqref{assumption_m_plus_one}}.
    Then, additionally to \eqref{Eq_convergences}, it holds 
    \begin{equation}\label{Eq_6}
     u_{\kappa,l}\,\to \, u ,\qquad \text{in }L^p(\Omega; L^r( [0,T] ; L^r (\T^d))),
    \end{equation}
    as $\kappa\to 0$ and $l\to \infty$, for any $p\in [1,\infty)$ and $r <(\alpha+m)+ \frac{2}{d}(\alpha+1)$. 
\end{col}
\begin{proof}
    We set $\zeta = \mathbf{1}_{\T^d}$ and take the stopping time  in \eqref{Eq_localized_energy_est_approx} to be 
    \begin{equation}
        \tau_j \, =\,  \inf  \biggl\{ t\in [0, T ] \, \bigg| \, 
        \sup_{t\le T}\|u_{\kappa,l}\|_{L^{\alpha+1}(\T^d)} \,+\, \int_{0}^t \int_{\T^d} 
        u_{\kappa,l}^{\alpha+m-2} |\nabla u_{\kappa,l}|^2
        \,\d x \,\d t  \,\ge\, j
        \biggr\} \,\wedge \, T
    \end{equation}
    and obtain
    \begin{align}\begin{split}
    \label{Eq_global_energy_est_approx}
&
 \E\biggl[\biggl(
        \sup_{t\le {\tau_j}} \int_{\T^d}  u_{\kappa,l}^{\alpha+1 }\,\d x \,+\, \int_0^{\tau_j} \int_{\T^d}    u_{\kappa,l}^{\alpha+m-2}|\nabla u_{\kappa,l}|^2\,\d x\, \d t \,+\,\kappa\int_0^{\tau_j} \int_{\T^d}u_{\kappa,l}^{\alpha-1}|\nabla u_{\kappa,l}|^2 \,\d x\, \d t \biggr)^p
        \biggr]^{1/p}
        \\&\quad \lesssim_{(\alpha,m,n,p,\phi,\psi)} \,\int_{\T^d}  u_0^{\alpha+1 }\,\d x \,+\,  \E\biggl[ \biggl(
        \int_0^{\tau_j} 
        \int_{\T^d} 
        u_{\kappa,l}^{\alpha+2n-1} 
        \,\d x
        \,\d t \biggr)^p
        \biggr]^{1/p}\\&\qquad +\, \frac{1}{l^{2n-2}}\E\biggl[
        \biggl(\int_0^{\tau_j} 
        \int_{\T^d} 
        u_{\kappa,l}^{\alpha+1} 
        \,\d x
        \,\d t \biggr)^p
        \biggr]^{1/p},
\end{split}
\end{align}
for $p\in [1,\infty)$.
We notice that if $l$ is sufficiently large we can absorb the last term on the right-hand side, so that we can disregard it. For the second term on the right-hand side we use instead \eqref{equation_n_m_together} to bound it by
\begin{align*}
    \E\biggl[\biggl(
        \int_0^{\tau_j} 
        \int_{\T^d} 
        u_{\kappa,l}^{\alpha+m} + u_{\kappa,l}
        \,\d x
        \,\d t  \biggr)^p
        \biggr]^{1/p} \,\leq \, \E\biggl[\biggl(
        \int_0^{\tau_j} 
        \int_{\T^d} 
        u_{\kappa,l}^{\alpha+m} 
        \,\d x
        \,\d t \biggr)^p
        \biggr]^{1/p} \,+\, T\|u_0\|_{L^1(\T^d)}.
\end{align*}
using additionally that mass is conserved. Then employing the Gagliardo--Nirenberg  \cite[Proposition A.1]{Passo_Giacomelli_Shishkov} and Young's inequality we estimate further
\begin{align*}
    \int_{\T^d} 
        u_{\kappa,l}^{\alpha+m} 
        \,\d x\,&\le \, \varepsilon  \int_{\T^d} 
        u_{\kappa,l}^{\alpha+m-2} |\nabla u_{\kappa,l}|^2
        \,\d x \,+\,
        C_{(\alpha, d, \varepsilon,m)} \biggl(\int_{\T^d} u_{\kappa,l}\,\d x\biggr)^{\alpha+m}\\&=\,  \varepsilon  \int_{\T^d} 
        u_{\kappa,l}^{\alpha+m-2} |\nabla u_{\kappa,l}|^2
        \,\d x \,+\,C_{(\alpha, d, \varepsilon,m)} \|u_0\|_{L^1(\T^d)}^{\alpha+m},
\end{align*}
where $\varepsilon>0$ is arbitrary.
Therefore, taking $\varepsilon$ sufficiently small and inserting this in \eqref{Eq_global_energy_est_approx} we end up with the bound
    \begin{align}
&
 \E\biggl[\biggl(
        \sup_{t\le {\tau_j}} \int_{\T^d}  u_{\kappa,l}^{\alpha+1 }\,\d x \,+\, \int_0^{\tau_j} \int_{\T^d}    u_{\kappa,l}^{\alpha+m-2}|\nabla u_{\kappa,l}|^2\,\d x\, \d t \,+\,\kappa\int_0^{\tau_j} \int_{\T^d}u_{\kappa,l}^{\alpha-1}|\nabla u_{\kappa,l}|^2 \,\d x\, \d t\biggr)^p
        \biggr]^{1/p}
        \\&\quad \lesssim_{(\alpha,m,n,p,\phi,\psi)} \,\int_{\T^d}  u_0^{\alpha+1 }\,\d x \,+\,T\bigl(\|u_0\|_{L^1(\T^d)} \,+\,
        \|u_0\|_{L^1(\T^d)}^{\alpha+m}\bigr).
\end{align}
Fatou's lemma allows us to pass to the limit $\tau_j\nearrow T$ based on the path properties $u_{\kappa,l}\in C([0,T]; L^{\alpha+1}(\T^d))$ and $u_{\kappa,l}^{(\alpha+m)/2} \in L^2([0,T]; H^1(\T^d))$   resulting in the bound
 \begin{align}\begin{split}\label{Eq_3}
&
 \E\biggl[\biggl(
        \sup_{t\le {T}} \int_{\T^d}  u_{\kappa,l}^{\alpha+1 }\,\d x \,+\, \int_0^{T} \int_{\T^d}    u_{\kappa,l}^{\alpha+m-2}|\nabla u_{\kappa,l}|^2\,\d x\, \d t \,+\,\kappa\int_0^{T} \int_{\T^d}u_{\kappa,l}^{\alpha-1}|\nabla u_{\kappa,l}|^2 \,\d x\, \d t\biggr)^p
        \biggr]^{1/p}
        \\&\quad \lesssim_{(\alpha,m,n,p,\phi,\psi)} \,\int_{\T^d}  u_0^{\alpha+1 }\,\d x \,+\,T\bigl(\|u_0\|_{L^1(\T^d)} \,+\,
        \|u_0\|_{L^1(\T^d)}^{\alpha+m}\bigr).\end{split}
\end{align}
for each $\kappa\in (0,\infty)$ and sufficiently large $l\in \N$.

\emph{Interpolating the estimates.} Next, we use again the Gagliardo--Nirenberg inequality as in \cite[Proposition A.1]{Passo_Giacomelli_Shishkov} with 
\begin{gather} \label{Eq_rel_q_theta1}
    \frac{1}{q} \,=\, \frac{d (\alpha+m)}{2(d(\alpha+m ) +2(\alpha+1)) } , \qquad \theta \,=\, \frac{2(\alpha+1 )  }{d (\alpha +m)+2(\alpha+1)}, \\ \frac{2}{q}\,=\, 1-\theta \,=\, \frac{d(\alpha+m)}{d(\alpha +m) + 2(\alpha+1)} \label{Eq_rel_q_theta2}
\end{gather}
so that 
\begin{align*}
    \frac{-d}{q} \,=\, \theta \Bigl(\frac{-d (\alpha + m)}{2(\alpha+1)}\Bigr) \,+\, (1-\theta) \Bigl( 1-\frac{d}{2}\Bigr)
\end{align*}
and hence
\begin{align*}
    \int_0^T  
    \|u_{\kappa,l}^{(\alpha+m)/2} \|_{L^q(\T^d) }^q
    \,\d t\,\lesssim_{(\alpha,d,m)} \, &\int_0^T \| u_{\kappa,l}^{(\alpha+m)/2}\|_{L^{2(\alpha+1)/(\alpha+m)}}^q \,\d t \\&+\,  \sup_{t\le T} \|u_{\kappa,l}^{(\alpha+m)/2}\|_{L^{2(\alpha+1)/(\alpha+m)}(\T^d)}^{\theta q} \int_0^T \| \nabla  u_{\kappa,l}^{(\alpha+m)/2} \|_{L^2(\T^d; \R^d)}^{(1-\theta)q} \,\d t.
\end{align*}
Using \eqref{Eq_rel_q_theta1}--\eqref{Eq_rel_q_theta2} and Young's inequality, we see that  the right-hand side is bounded by
\begin{align}\label{Eq_2}
    (T+1)\biggl(\sup_{t\le T} \int_{\T^d} u_{\kappa,l}^{\alpha+1} \,\d x\biggr)^{\frac{d(\alpha+m) +2 (\alpha+1) }{d(\alpha+1)}} \,+\, \biggl( \int_0^T \int_{\T^d} 
    u_{\kappa,l}^{\alpha+m-2}|\nabla u_{\kappa,l}|^2
    \,
    \d x\d t \biggr)^{\frac{d(\alpha+m) +2(\alpha+1)}{d(\alpha+m)}}.
\end{align}

\emph{Conclusion.} As arbitrary moments of \eqref{Eq_2} are uniformly in $\kappa,l$ estimated by \eqref{Eq_3} with sufficiently large $p$, we conclude that
\[
\sup_{\kappa \in (0,\infty) , l\in \N }\E \biggl[\biggl( \int_0^T \int_{\T^d} u_{\kappa,l}^{(\alpha+m)+ \frac{2}{d}(\alpha+1)} \,\d x \, \d t\biggr)^\frac{p+1}{(\alpha+m)+ \frac{2}{d}(\alpha+1)}\biggr]\,<\,\infty, \qquad p\in [1,\infty),
\]
and consequently, up to taking a subsequence, we have \[u_{\kappa,l} \rightharpoonup v\qquad \text{in }L^{p+1}( \Omega ; L^{(\alpha+m)+ \frac{2}{d}(\alpha+1)}(  [0,T] \times \T^d )).\]
Since by Vitali's theorem the above bound together with \eqref{Eq_convergences} also yields $u_{\kappa,l} \to u$ in $L^1(\Omega \times [0,T] \times \T^d)$ we conclude that $u=v$. The claimed convergence \eqref{Eq_6} follows for this subsequence by interpolation and by a subsequence-subsequence argument also for  the original  $\kappa\to 0$, $l\to \infty$.
\end{proof}
\begin{col}[Localized energy estimate for \eqref{Eq_SPME_intro}]\label{Cor_loc_energy}Let $\zeta\in C^2(\T^d)$, $p\in [1,\infty)$ and $\tau\le T$ a stopping time. Then it holds
\begin{align}\begin{split}
    \label{Eq_localized_energy_est}
&
 \E\biggl[\biggl(
        \sup_{t\le \tau} \int_{\T^d} \zeta^2 u^{\alpha+1 }\,\d x \,+\, \int_0^\tau \int_{\T^d}   \zeta^2 | \nabla (u^{(\alpha+m )/2 })|^2\,\d x\, \d t \biggr)^p
        \biggr]^{1/p}
        \\&\quad \lesssim_{(\alpha,m,n,p,\psi)} \,\int_{\T^d} \zeta^2 u_0^{\alpha+1 }\,\d x \,+\, \E \biggl[\biggl(\int_0^\tau \int_{\T^d} u^{\alpha+m}|\nabla \zeta|^2 \,\d x\,\d t \biggr)^p
        \biggr]^{1/p}
        \\&\qquad +\, \E\biggl[\biggl(
        \int_0^\tau 
        \int_{\T^d} 
        u^{\alpha+2n-1} \bigl(|\zeta \Delta  \zeta | + |\nabla \zeta|^2 +\zeta^2 \bigr)
        \,\d x
        \,\d t \biggr)^p
        \biggr]^{1/p}
\end{split}
\end{align}
for the unique kinetic solution 
$u$ to \eqref{Eq_SPME_intro}  with initial value $u_0$ {satisfying \eqref{assumption_m_plus_one}}.
\end{col}
\begin{proof}
    We deduce from \eqref{Eq_convergences} and Corollary \ref{cor_add_conv} that 
    \begin{align}\begin{split}\label{Eq_20}
        \mathbf{1}_{[0,\tau]} \zeta^2 u_{\kappa,l}^{\alpha+1} \,\to\, 
        \mathbf{1}_{[0,\tau]} \zeta^2 u^{\alpha+1},\qquad &\P\otimes \d t \otimes \d x\text{-a.e.},\\
        \mathbf{1}_{[0,\tau]} \zeta \nabla (u_{\kappa,l}^{(\alpha+m)/2}) \,\to\, 
        \mathbf{1}_{[0,\tau]}\zeta \nabla  (u^{(\alpha+m)/2}),\qquad & \text{in }(L^2([0,T]\times\T^d; \R^d) , d_{w.b}), \,\P\text{-a.s.},        \end{split}
    \end{align}
    up to passing to a subsequence.
    To proceed, we notice that the functional $\|\cdot \|_{L^2([0,T]\times\T^d; \R^d)}$ is lower semicontinuous with respect to $d_{w,b}$: Indeed, for $v_k \to  v$ in $d_{w,b}$, we may consider  a subsequence $(v_{k_j})_j$ for which 
    $$
    \lim_{j\to\infty}\|v_{k_j} \|_{L^2([0,T]\times\T^d; \R^d)}\,=\, 
    \liminf_{k\to\infty} \|v_k \|_{L^2([0,T]\times\T^d; \R^d)}.$$ 
    Then, the above is either $\infty$ and the estimate
    \begin{equation}\label{eqn35}
    \|v\|_{L^2([0,T]\times\T^d; \R^d)} \,\le\, \liminf_{k\to\infty} \|v_k \|_{L^2([0,T]\times\T^d; \R^d)}
    \end{equation}
    is obvious, or $(v_{k_j})_j$ is norm-bounded. But in the latter case, $d_{w,b}$-convergence is just weak convergence and \eqref{eqn35} follows as well. Thus, using also Fatou's lemma, we find that
    \begin{align*}&
        \E\biggl[\biggl(
        \sup_{t\le \tau} \int_{\T^d} \zeta^2 u^{\alpha+1 }\,\d x \,+\, \int_0^\tau \int_{\T^d}   \zeta^2 | \nabla (u^{(\alpha+m )/2 })|^2\,\d x\, \d t \biggr)^p
        \biggr]^{1/p}
        \\&\quad \le\,\liminf_{l\to\infty} \E\biggl[\biggl(
        \sup_{t\le \tau} \int_{\T^d} \zeta^2 u_{\kappa,l}^{\alpha+1 }\,\d x \,+\, \int_0^\tau \int_{\T^d}   \zeta^2 u_{\kappa,l}^{\alpha+m-2}|\nabla u_{\kappa,l}|^2\,\d x\, \d t \biggr)^p
        \biggr]^{1/p}.
    \end{align*}
    The right-hand side is estimated for fixed $l$ up to a multiplicative constant by 
    \begin{align*}
        &\int_{\T^d} \zeta^2 u_0^{\alpha+1 }\,\d x \,+\, \E \biggl[\biggl(\int_0^\tau \int_{\T^d} u_{\kappa,l}^{\alpha+m}|\nabla \zeta|^2 \,\d x\,\d t \,+\,
        \kappa\int_0^\tau  \int_{\T^d}u_{\kappa,l}^{\alpha+1}|\nabla \zeta|^2\,\d x\, \d t\biggr)^p
        \biggr]^{1/p}
        \\&\quad +\, \E\biggl[\biggl(
        \int_0^\tau 
        \int_{\T^d} 
        u_{\kappa,l}^{\alpha+2n-1} \bigl(|\zeta \Delta  \zeta| + |\nabla \zeta|^2 +\zeta^2 \bigr)
        \,\d x
        \,\d t \biggr)^p
        \biggr]^{1/p}\\&\quad +\, \frac{1}{l^{2n-2}}\E\biggl[\biggl(
        \int_0^\tau 
        \int_{\T^d} 
        u_{\kappa,l}^{\alpha+1} \bigl(|\zeta \Delta  \zeta| + |\nabla \zeta|^2 +\zeta^2  \bigr)
        \,\d x
        \,\d t \biggr)^p
        \biggr]^{1/p}
    \end{align*}
    due to Lemma \ref{Lemma_localized_approx}. This converges to the right-hand side of \eqref{Eq_localized_energy_est} as $l\to\infty$ by Corollary \ref{cor_add_conv}.
\end{proof}
As a byproduct of our proof of the preceding localized energy estimate we also obtain a global bound on the kinetic solution $u$, which we record for later use.
\begin{col}[Global energy estimate for \eqref{Eq_SPME_intro}]\label{cor_global_energy}
    The unique kinetic solution to \eqref{Eq_SPME_intro} with initial value $u_0$ { satisfying \eqref{assumption_m_plus_one} admits the estimate} 
    \begin{align}
        \begin{split}\label{Eq_glob_enegry_est}
&
 \E\biggl[\biggl(
        \sup_{t\le {T}} \int_{\T^d}  u^{\alpha+1 }\,\d x \,+\, \int_0^{T} \int_{\T^d}   | \nabla (u^{(\alpha+m )/2 })|^2\,\d x\, \d t \biggr)^p
        \biggr]^{1/p}
        \\&\quad \lesssim_{(\alpha,m,n,p,\psi)} \,\int_{\T^d}  u_0^{\alpha+1 }\,\d x \,+\,T\bigl(\|u_0\|_{L^1(\T^d)} \,+\,
        \|u_0\|_{L^1(\T^d)}^{\alpha+m}\bigr).
        \end{split}
    \end{align}
\end{col}
\begin{proof}
    This follows from \eqref{Eq_convergences}, \eqref{Eq_3} and \eqref{Eq_20} with $\tau=T$.
\end{proof}
{Finally, we argue that the two preveding results hold also without the additional assumption \eqref{assumption_m_plus_one}.

\begin{prop}\label{prop_generalization}
    The conclusions of Corollary \ref{Cor_loc_energy} and Corollary \ref{cor_global_energy} also hold without the additional assumption \eqref{assumption_m_plus_one} on $u_0$.
\end{prop}
\begin{proof}
     For each $R \in \N$ the truncated $u_{0,R} = u_0\wedge R$ clearly satisfies \eqref{assumption_m_plus_one}, so that  Corollary \ref{Cor_loc_energy} and Corollary \ref{cor_global_energy} are valid for the kinetic solution $u_R$ to \eqref{Eq_SPME_intro} with said initial value. To generalize these properties to $u$, we essentially have to repeat the above approximation argument, for which we leverage the estimate 
    \begin{equation} \label{eqn_pathwise_contr}
        \| u_R-  u \|_{L^\infty (\Omega\times [0,T];L^1(\T^d )) } \le \|u_{0,R} - u_0\|_{L^1(\T^d)}  
    \end{equation}
    from \cite[Theorem 4.6]{Fehrman_Gess_ARMA}. We notice that the right-hand side of \eqref{eqn_pathwise_contr} clearly tends to $0$ as $R\to\infty$. The uniform bound \eqref{Eq_glob_enegry_est}, with $u$ replaced by $u_R$, implies then by the interpolation argument from the proof of Corollary \ref{cor_add_conv}  that 
    \begin{equation}\label{Eq_67}
     u_{R}\,\to \, u ,\qquad \text{in }L^p(\Omega; L^r( [0,T] ; L^r (\T^d))),
    \end{equation}
    for all $p\in [1,\infty)$ and $r <(\alpha+m)+ \frac{2}{d}(\alpha+1)$.  
    This, together with the lower weak semicontinuity of norms with respect to weak convergence suffices  to deduce \eqref{Eq_localized_energy_est} and \eqref{Eq_glob_enegry_est} for $u$ itself.

\end{proof}
}
\section{Proof of the Main Result}\label{Sec_FSOP}

In this section, we prove that kinetic solutions to \eqref{Eq_SPME_intro} have finite speed of propagation as well as that waiting time phenomena will occur if the profile of the initial data is sufficiently flat, i.e., Theorem \ref{Thm_fsop}.  To this end, we denote dependence on a constant by $\lesssim$, where we allow in this section also for  a dependence on the model parameters $(\alpha,m,n,\psi,d)$ without further indication. For a given $x_0\in \mathbb T^d$ we omit moreover the dependence of $R_0$, $\tau_{R_0}$, and $\tau_r$ (as defined in and above Definition \ref{defn_fsop}) on $x_0$ to ease notation even further. Accordingly, all balls will be centered at $x_0$ unless otherwise specified.

The main idea of the proof is to employ the filtering technique proposed in \cite{Fischer_Grun_FSOP}. 
For this, we derive from the localized energy estimate of Corollary \ref{Cor_loc_energy}{, see also Proposition \ref{prop_generalization},}  a Stampacchia-type inequality based on the  iteration trick stated in Lemma \ref{lemma-iteration-trick}. This enables us to apply a simple version of said filtering argument and bypass difficulties arising from having multiple terms on the right-hand side of \eqref{Eq_localized_energy_est}, giving an alternative approach to \cite{GK_fsop3} where similar challenges are addressed. We begin by proving the Stampacchia-type estimate. For numbers $R,s, \delta$ such that $1/2 > R \geq   s+\delta  >0$, consider functions $\zeta_{s,\delta}\in C^2(\mathbb T^d)$ characterized by the following properties: 
\begin{enumerate}[label = (L\arabic*)]\label{cut-off-conditions}
\item\label{L1} $0\leq \zeta_{s,\delta}(x) \leq 1$ for all $x\in \mathbb T^d$. 
\item\label{L2} $\zeta_{s,\delta} \equiv 1$ on $B_{R-s-\delta}$.
\item\label{L3}  $\zeta_{s,\delta} \equiv 0$ on the set $ \mathbb T^d \setminus B_{R-s}$.
\item\label{L4} There exists $c>0$ independent of $s,\delta$ so that $\abs{\nabla\zeta_{s,\delta}(x)} \leq \frac{c}{\delta}$ and  $ \abs{D^2\zeta_{s,\delta}(x)} \leq \frac{c}{\delta^2}$ for all $x\in \T^d$.
\end{enumerate}  
  Using \ref{L1}--\ref{L4}, we deduce our main technical lemma from the localized energy estimate \eqref{Eq_localized_energy_est}:
\begin{lemma}
    \label{lemma-iteration-expectation}
    Let $u$ be the unique kinetic solution to \eqref{Eq_SPME_intro} with initial value $u_0$, $x_0 \in \mathbb T^d$ as constructed in \cite{Fehrman_Gess_ARMA}. For $i\in \left\{1,2,3\right\}$, there exist constants $\ell_i \geq 0$ and $\kappa_i>1$ such that for any stopping time $\tau \geq 0$, the following holds for every $\frac{1}{2}>R\geq s+\delta >s \geq 0$:
    \begin{equation}
        \label{eq-stampacchia-expectation}
        \e{ \sup_{(0,\tau)} \int_{B_{R-s-\delta}} u^{\alpha+1} } \lesssim\int_{B_{R-s}} u_{0}^{\alpha+1}+ \sum_{i=1}^3\e{\frac{\tau}{\delta^{\ell_i}} \left( \sup_{(0,\tau)}\int_{B_{R-s}} u^{\alpha+1}\right)^{\kappa_i} }. 
    \end{equation}
\end{lemma}

\begin{remark}\label{remark-theta-value}
  For explicit formulas for $\kappa_i$ and $\ell_i$, $i\in\{1,2,3\}$, we refer to Lemma~\ref{lemma-algebra} in Appendix~\ref{App_C}, more precisely to equations \eqref{GG-2} and \eqref{GG-1}. We emphasize that the parameter $\theta=\frac{2(\alpha+1)+d(m-1)}{2(n-1)}$ given in Theorem~\ref{Thm_fsop}, which enters the flatness condition of initial data to guarantee the occurrence of a waiting time phenomenon, is obtained by maximization of the quotients $\tfrac{\ell_i}{\kappa_i-1}$, i.e.
  \[
        \theta = \max_{i=1,2,3}\left\{ \frac{\ell_i}{\kappa_i-1} \right\}.
  \]
  We emphasize that the parameters  $\kappa_i$, $i\in\{1,2,3\}$,  are the same as those used  in the statement of Theorem~\ref{Thm_fsop} (see Lemma \ref{lemma-algebra} for a proof). 
\end{remark}

\begin{proof}
     To begin, fix a stopping time $\tau\geq 0$, $x_0\in \mathbb T^d$, along with $R<1/2$. For arbitrary but fixed $s,\delta$ such that $0\leq s<s+\delta \leq R$, use the function $\zeta_{s+\frac{\delta}{2},\frac{\delta}{2}}$ in the energy estimate \eqref{Eq_localized_energy_est}. Using properties \ref{L1}-\ref{L4}, one derives  
\begin{equation}
    \label{eq-recursion-proof1}
    \e{\sup_{(0,\tau)} \int_{B_{R-s-\delta}} u^{\alpha+1} + \int_0^\tau\int_{B_{R-s-\delta}} \abs{\nabla u^{\frac{\alpha+m}{2}}}^2 } \lesssim \int_{B_{R-s-\frac{\delta}{2}}} u_0^{\alpha+1} + \sum_{i=1}^3 \left(\frac{2}{\delta}\right)^{L_i}\e{\int_0^\tau \int_{B_{R-s-\frac{\delta}{2}}} u^{p_i} },
\end{equation}
where the parameters $L_i,p_i$ are given by 
\begin{gather}\label{eq:GG-137}
  \begin{split}
  L =(L_1,L_2, L_3):= (2,2,0), \\
   p=(p_1, p_2, p_3):= (\alpha+m, \alpha+2n-1, \alpha+2n-1).
\end{split}
 \end{gather}
Note
\[
    \int_{B_{R-s-\frac{\delta}{2}}} u_0^{\alpha+1} \leq \int_{B_{R-s}} u_0^{\alpha+1}.
\]
For the other terms, we apply the Gagliardo--Nirenberg inequality \eqref{eq-GN} along with Young's inequality to see
\begin{align*}
    \left(\frac{2}{\delta}\right)^{L_i}\e{\int_0^\tau \int_{B_{R-s-\frac{\delta}{2}}} u^{p_i} } & \leq \left(\frac{2}{\delta}\right)^{L_i}\e{\int_0^\tau \int \zeta_{s+\frac{\delta}{4},\frac{\delta}{4}}^{\frac{4p_i}{\alpha+m}}u^{p_i} } \\
    &\lesssim \left(\frac{2}{\delta}\right)^{L_i} \e{ \int_0^\tau \left( \int \abs{\nabla\left( \zeta_{s+\frac{\delta}{4},\frac{\delta}{4}}^2 \, u^{\frac{\alpha+m}{2}} \right)}^2\right)^{\frac{p_i\theta_i}{\alpha+m}} \left( \int \zeta_{s+\frac{\delta}{4},\frac{\delta}{4}}^{\frac{4(\alpha+1)}{\alpha+m}} \, u^{\alpha+1} \right)^{(1-\theta_i)\frac{p_i}{\alpha+1}} } \\
    & \lesssim \epsilon \e{\int_0^\tau \int  \abs{\nabla\left( \zeta_{s+\frac{\delta}{4},\frac{\delta}{4}}^2 \, u^{\frac{\alpha+m}{2}} \right)}^2 } \\
    &\hspace{2cm}+ C_\epsilon \left( \frac{2}{\delta} \right)^{\frac{L_i(\alpha+m)}{\alpha+m - p_i\theta_i}} \e{ \int_0^\tau\left( \int_{B_{R-s}} u^{\alpha+1}  \right)^{(1-\theta_i)\frac{p_i}{\alpha+1}\left( \frac{\alpha+m}{\alpha+m-p_i\theta_i}\right) }} \\
    & \lesssim \epsilon\left( \frac{4}{\delta} \right)^{2} \e{ \int_0^\tau \int_{B_{R-s-\frac{\delta}{4}}} u^{\alpha+m} } + \epsilon \e{ \int_0^\tau \int_{B_{R-s}} \abs{\nabla u^{\frac{\alpha+m}{2}}}^2 }  \\
    & \hspace{2cm} + C_\epsilon \left( \frac{2}{\delta} \right)^{\frac{L_i(\alpha+m)}{\alpha+m - p_i\theta_i}} \e{ \int_0^\tau\left( \int_{B_{R-s}} u^{\alpha+1}  \right)^{(1-\theta_i)\frac{p_i}{\alpha+1}\left( \frac{\alpha+m}{\alpha+m-p_i\theta_i}\right)}}.
\end{align*}
We pause to explain the steps of the above estimate. The first inequality follows by seeing that $\zeta_{s+\frac{\delta}{4},\frac{\delta}{4}}$ is identically one on $B_{R-s-\frac{\delta}{2}}$ and is supported on $B_{R-s-\frac{\delta}{4}}$. The exponent attached is for aesthetic purposes. The second step is simply the application of \eqref{eq-GN}, and the third step is by Young's inequality and the fact that $B_{R-s-\frac{\delta}{4}}\subset B_{R-s}$. The last step uses convexity and \ref{L4}. We note that $\theta_i$ as given in \eqref{eq-GN} satisfies 
\begin{equation}\label{GG-3}
    \theta_i = \frac{\frac{\alpha+m}{2(\alpha+1)} - \frac{\alpha+m}{2p_i}}{\frac{1}{d}-\frac{1}{2}+\frac{\alpha+m}{2(\alpha+1)}}.
\end{equation}
Using this along with \eqref{equation_n_m_together}, one may indeed verify that 
\[
    \frac{p_i\theta_i}{\alpha+m}<1
\]
for each $i$. We now wish to apply Lemma \eqref{lemma-iteration-trick}. To that end, let 
\begin{equation}\label{eq-iteration-lemma-setup}\begin{split}
    V(s+\delta) &:= \e{\sup_{(0,\tau)} \int_{B_{R-s-\delta}} u^{\alpha+1}} ,\\
    E(s+\delta) &:= \e{\int_0^\tau\int_{B_{R-s-\delta}} \abs{\nabla u^{\frac{\alpha+m}{2}}}^2 } ,\\
    A(s) &:= C\int_{B_{R-s}} u_0^{\alpha+1} ,\\
    U_i\left(s+\frac{\delta}{2}\right) &:= C\e{\int_0^\tau \int_{B_{R-s-\frac{\delta}{2}}} u^{p_i} } ,\\
    F_j(s) &:= C\e{ \int_0^\tau\left( \int_{B_{R-s}} u^{\alpha+1}  \right)^{(1-\theta_i)\frac{p_i}{\alpha+1}\left( \frac{\alpha+m}{\alpha+m-p_i\theta_i}\right)}}, \\
    \ell_i &:= \frac{L_i(\alpha+m)}{\alpha+m - p_i\theta_i}.
    \end{split}
\end{equation}
The presence of the constant $C>0$ in the above definitions is due to the implicit constant in the preceding estimates. It must therefore be included in the definitions above to remain consistent with the statement of Lemma \ref{lemma-iteration-trick}. Then, the above estimate in conjunction with \eqref{eq-recursion-proof1} yields a chain of inequalities of the form \eqref{eq-iteration-condition-general}; hence, there exists a constant (once again only depending on the model parameters) such that 
\[
V(s+\delta) \lesssim A(s) + \sum_{i=1}^3 \delta^{-\ell_i}F_i(s). 
\]
The last step is to let 
\begin{equation}\label{eq-def-kappa}
    \kappa_i := (1-\theta_i)\frac{p_i}{\alpha+1}\left( \frac{\alpha+m}{\alpha+m-p_i\theta_i}\right)
\end{equation}
and use H\"older's inequality in time to see that for any $i$, 
\[
    F_i(s) \leq \e{ \tau V(s)^{\kappa_i}}. 
\]
\end{proof}

We are now ready to prove the main result of this section, which implies Theorem \ref{Thm_fsop}.
\begin{thm}[Vanishing of the Indicator Function]\label{thm-fsop-quant}
    Let $u$ be the unique kinetic solution to \eqref{Eq_SPME_intro} with initial value $u_0$. For a given $x_0\in \mathbb T^d\setminus\mathrm{supp}(u_0)$ define the process $G:[0,T] \times \Omega\times \left[0,\frac{1}{2}\right) \ra \R_{\geq 0}$ by 
    \begin{equation}\label{equation_for_G}
        G(t,\omega, r): = \sup_{s\in (0,t)} \, \int_{B_r(x_0)} u_s^{\alpha+1}.
    \end{equation}
    Then for any $r< R_0$, we have 
    \[
        \lim_{t\ra 0} \mathbb P\left( G(t,r) > 0 \right) = 0.
    \]
    If in addition the assumption \eqref{eq_flatness} of Theorem \ref{Thm_fsop} holds, then 
    \[
         \lim_{t\ra 0} \mathbb P\left( G(t,R_0) > 0 \right) = 0.
    \]
\end{thm}

\noindent Before proving  Theorem~\ref{thm-fsop-quant}, we note that it immediately implies Theorem~\ref{Thm_fsop} .

\begin{proof}[Proof of Theorem \ref{Thm_fsop}] Let $x_0\in \mathbb T^d\setminus \mathrm{supp}(u_0)$ and let $r\in (0,R_0(x_0))$. By the prior result we have
\begin{equation}\label{goal}
\lim_{t\ra 0} \mathbb{P}\left( G(t,r)>0 \right)= 0.
\end{equation}
Suppose now that 
\[
\omega \in \left\{ \tau_r = 0\right\}.
\]
Then,  for any $t>0$, 
\[
 \omega \in \left\{ G(t,r) > 0 \right\}
\]
and so 
\[
    \left\{ \tau_r = 0 \right\} \subset \bigcap_{t>0}\left\{ G(t,r) > 0 \right\}.
  \]
By Theorem \ref{thm-fsop-quant}, the set on the right has probability zero, hence $\mathbb P\left( \tau_r = 0 \right) = 0$, since it is an intersection of sets of arbitrarily small probability, which completes the proof of FSOP.

Under the assumption $\eqref{eq_flatness}$, we also have by Theorem \ref{thm-fsop-quant} that 
\[
    \lim_{t\ra 0}\mathbb P\left( G(t,R_0)>0 \right) = 0.
\]
The same argument as above shows that $\mathbb P\left( \tau_{R_0} = 0 \right) = 0$, which proves the existence of a WTP associated with $B_{R_0}$.

\end{proof}

\begin{proof}[Proof of Theorem \ref{thm-fsop-quant}]
Fix $t\in (0,T]$ and $r\in (0,R_0)$. The proof relies on the following decomposition of $\{G(t,r) > 0\}$, which was originally suggested in \cite{Fischer_Grun_FSOP}: 
Let $(\mu_k)_{k\in \N}$ be a sequence such that $\mu_k\searrow 0$. Define the sequence $(r_k)_{k\in \N}$ by 
    \[
        r_k = r + \frac{R_0-r}{2^k}, 
    \]
    and set $r_0 := R_0$. For such sequences, one may define stopping times $(\tau_k)_{k\in \N}$ given by $\tau_1 = t$ and 
    \begin{align}\label{def_tau_k}
       \tau_k = \inf\{s\in [0, \tau_{k-1}], \quad G(s,r_{k-1}) > \mu_{k-1}\}.
    \end{align} 
    An elementary argument, that the above defines a stopping time is the content of Lemma \ref{Lemma_appendix_D}. 
    One then may use the inclusion
    \begin{align}\label{eqn3}
        \left\{ G(t,r) > 0  \right\} \subset  \bigcup_{k\in \N} \{ G(\tau_k,r_k)>\mu_k \},
    \end{align}
    for a proof see \cite{Fischer_Grun_FSOP}.
From \eqref{eqn3} we have  
\begin{equation}\label{probability_main} 
\mathbb{P}\left(G(t,r) > 0 \right) \leq \sum_{k=1}^\infty \mathbb{P}\left( G(\tau_k,r_k) > \mu_{k} \right) .
\end{equation}
Applying Markov's inequality and \eqref{eq-stampacchia-expectation} with $R = R_0$, $s= R_0-r_{k+1}$, $\delta = r_{k-1}-r_{k} = 2^{-k}(R_0-r)$, and $\tau=\tau_k$ yields then for $k\geq 2$:
\begin{align*}
    \mathbb{P}\left( G(\tau_k,r_k)  > \mu_{k} \right) &\leq \mu_k\inv\e{G(\tau_k,r_k)} = \mu_k\inv  \e{ \sup_{(0,\tau_k)} \int_{B_{r_k}} u^{\alpha+1} }  \\
    & \lesssim \mu_k\inv \left( \int_{B_{r_{k-1}}} u_{0}^{\alpha+1}+ \sum_{i=1}^3\e{\frac{2^{\ell_i k}\cdot \tau_k}{(R_0-r)^{\ell_i}} \left( \sup_{(0,\tau_k)}\int_{B_{r_{k-1}}} u^{\alpha+1}\right)^{\kappa_i} } \right) \\
    & \lesssim \, t\cdot \sum_{i=1}^3 \frac{\mu_k\inv\cdot 2^{\ell_i k}}{(R_0-r)^{\ell_i}} \e{G(\tau_k,r_{k-1})^{\kappa_i}} \\
    & \lesssim \, t\cdot \sum_{i=1}^3 \frac{\mu_k\inv \mu_{k-1}^{\kappa_i}\cdot 2^{\ell_i k}}{(R_0-r)^{\ell_i}}.
\end{align*}
In the second inequality, we have simply used that $\tau_k \leq t$ and the fact that $r_k< R_0$, which causes the initial term to vanish. For the $k=1$ term, we also use \eqref{eq-stampacchia-expectation} to deduce 
\[
    \mathbb{P}\left( G(t,r_1)>\mu_1 \right)  \lesssim \, t\cdot \sum_{i=1}^3 \frac{\mu_1\inv\cdot 2^{\ell_i}}{(R_0-r)^{\ell_i}} \e{G(t,R_0)^{\kappa_i}}.
\]
By Corollary \ref{cor_global_energy}{, see also Proposition \ref{prop_generalization}}, we have
\[
     \e{G(t,R_0)^{\kappa_i}} \leq C_{(\alpha,m,n,\psi,d)}\left( \,\int_{\T^d}  u_0^{\alpha+1 }\,\d x \,+\,T\bigl(\|u_0\|_{L^1(\T^d)} \,+\,
        \|u_0\|_{L^1(\T^d)}^{\alpha+m}\bigr) \right)^{\kappa_i} =: C_i.
\]
Combining the previous three estimates and substituting into \eqref{probability_main}, we have 
\begin{equation}\label{eq-probability-infinite-sum}
     \mathbb{P}\left( G(t,r)  > 0 \right)  \lesssim \, t\cdot \sum_{i=1}^3  \frac{C_i\, \mu_1\inv\cdot 2^{\ell_i}}{(R_0-r)^{\ell_i}} +  \sum_{k=2}^\infty\frac{\mu_k\inv \mu_{k-1}^{\kappa_i}\cdot 2^{\ell_i k}}{(R_0-r)^{\ell_i}}.
\end{equation}
Choosing now $\mu_ k = \mu \nu^{k-1}$ with $\mu >0$ and $\nu\in (0,1)$, \eqref{eq-probability-infinite-sum} becomes 
\begin{equation}\label{eq-probability-infinite-sum2}
     \mathbb{P}\left( G(t,r)  > 0 \right)  \lesssim\, t\cdot \sum_{i=1}^3  \frac{C_i\, \mu\inv\cdot 2^{\ell_i}}{(R_0-r)^{\ell_i}} +  \frac{\mu^{\kappa_i-1}}{\nu^{2\kappa_i-1}}\sum_{k=2}^\infty\frac{\nu^{(\kappa_i-1)k}\cdot 2^{\ell_i k}}{(R_0-r)^{\ell_i}}.
\end{equation}
Then for $\nu < 2^{-\frac{\ell_i}{\kappa_i-1}}$ for each $i$, each of the infinite series on the right converge. Seeing that the terms inside the sum do not depend on $t$, the first part of the theorem is  thereby  established. 

For the second part, the argument is, in principle, the same, but the initial term will cause a non-trivial contribution. Assume now that \eqref{eq_flatness} holds, i.e., there exists $r_0$ and $f\in \mathscr C_{2r_0, \kappa}$ (recall the definition of $\mathscr C_{2r_0, \kappa} $ above Theorem \eqref{Thm_fsop}) such that 
\begin{equation}\label{GG-6}
    \sup_{r\in (0,r_0)} \ \frac{1}{f(r)r^{\theta}} \ \int_{B_{R_0+r}(x_0)\setminus B_{R_0}(x_0)} u_0^{\alpha+1} = S< \infty.
\end{equation}
Just as in the previous part of the proof, we take a sequence $(\mu_k)_{k\in \N}$ such that $\mu_k \searrow 0$ to be specified later. This time, we consider a new sequence $(\xi_k)_{k\in \N}$ given by 
\[
    \xi_k = R_0 + \frac{r_0}{2^k},
\]
where we also impose $\xi_0 = R_0+r_0$.
By the same line of argumentation, we have (using the same definitions of $\tau_k$ up to replacement of $r_k$ with $\xi_k$): 
\begin{equation}
    \label{eq-probability-main-wtp}
    \mathbb P\left( G(t,R_0)>0 \right) \leq \sum_{k=1}^\infty \mathbb P\left( G(\tau_k, \xi_k)>\mu_k \right).
\end{equation}
Applying Markov's inequality, the identity ~\eqref{GG-6}, and \eqref{eq-stampacchia-expectation} with $R = R_0+r_0$, $s= R_0+r_0-\xi_{k-1}$, $\delta = \xi_{k-1}-\xi_{k} = 2^{-k}r_0$, and $\tau=\tau_k$, we obtain for any $k\geq 2$ that:
\begin{align*}
    \mathbb{P}\left( G(\tau_k,\xi_k)  > \mu_{k} \right) &\leq \mu_k\inv\e{G(\tau_k,\xi_k)} =  \mu_k\inv\e{ \sup_{(0,\tau_k)} \int_{B_{\xi_k}} u^{\alpha+1} }  \\
    & \lesssim \mu_k\inv \left( \int_{B_{\xi_{k-1}}} u_{0}^{\alpha+1}+ \sum_{i=1}^3\e{\frac{2^{\ell_i k}\cdot \tau_k}{r_0^{\ell_i}} \left( \sup_{(0,\tau_k)}\int_{B_{\xi_{k-1}}}u^{\alpha+1}\right)^{\kappa_i} } \right) \\
    & \lesssim \mu_k\inv \left( S\cdot \left( \frac{r_0}{2^{k-1}} \right)^\theta\cdot f\left( \frac{r_0}{2^{k-1}}\right) + \, t \sum_{i=1}^3\frac{\mu_{k-1}^{\kappa_i} 2^{\ell_i k}\cdot}{r_0^{\ell_i}}  \right).
\end{align*}
We also have for the initial term
\begin{align*}
    \mathbb{P}\left( G(t,\xi_1)>\mu_1 \right)  & \lesssim \mu_1\inv\int_{B_{\xi_0}} u_0^{\alpha+1}  + \, t\cdot \sum_{i=1}^3 \frac{\mu_1\inv\cdot 2^{\ell_i}}{r_0^{\ell_i}} \e{G(t,\xi_0)^{\kappa_i}} \\
    & \lesssim \mu_1\inv S \cdot r_0^\theta \cdot f(r_0)   + \, t\cdot \sum_{i=1}^3 \frac{C_i \,\mu_1\inv\cdot 2^{\ell_i}}{r_0^{\ell_i}}. 
\end{align*}
Here, $C_i$ is the same constant as earlier in the proof. Plugging the previous two estimates into \eqref{eq-probability-main-wtp}, we obtain 
\begin{equation}
    \label{eq-probability-main-wtp-sums}
    \mathbb P\left( G(t,R_0)>0 \right) \lesssim S\, (2r_0)^\theta \, \sum_{k=1}^\infty \mu_k\inv 2^{-\theta k} f\left( \frac{r_0}{2^{k-1}} \right) + t \cdot \sum_{i=1}^3 \left( \frac{C_i\, \mu_1\inv\cdot 2^{\ell_i}}{r_0^{\ell_i}} +  \sum_{k=2}^\infty\frac{\mu_k\inv \mu_{k-1}^{\kappa_i}\cdot 2^{\ell_i k}}{r_0^{\ell_i}} \right).
\end{equation}
Since $f \in \mathscr C_{2r_0, \kappa}$, there exists $(a_k)_{k\in \N}$ such that 
\[
    \sum_{k=1}^\infty a_{k-1}^{\kappa_i}a_k\inv < +\infty, \quad \sum_{k=1}^\infty a_k\inv f\left( 2r_0\cdot 2^{-k} \right) < +\infty.
\]
To exploit this, choose 
\[
    \mu_k = \mu \cdot 2^{-\theta k}\cdot a_k.
  \]
  Before we proceed, we use Lemma~\ref{lemma-algebra} which states that the parameter $\theta$ appearing in assumption~\eqref{GG-999} of Theorem~\ref{Thm_fsop} satisfies
  \begin{equation}
    \label{GG-7}
    \theta=\frac{2(\alpha+1)+d(m-1)}{2(n-1)}=\max_{i\in\{1,2,3\}}\left\{\frac{\ell_i}{\kappa_i-1}\right\}.
  \end{equation}
Plugging this into \eqref{eq-probability-main-wtp-sums}, we have 
\begin{align*}
    \mathbb P\left( G(t,R_0)>0 \right) & \lesssim\, \mu\inv S\, (2r_0)^\theta \, \sum_{k=1}^\infty a_k\inv f\left( 2r_0 \cdot 2^{-k} \right) \\
    &\hspace{1cm} + t \cdot \sum_{i=1}^3 \left( \frac{C_i\, \mu\inv\cdot a_1\inv \cdot 2^{\ell_i+\theta}}{r_0^{\ell_i}} +  \frac{2^{\theta\kappa_i}\, \mu^{\kappa_i -1}}{r_0^{\ell_i}}\sum_{k=2}^\infty a_{k-1}^{\kappa_i}a_k\inv\right),
\end{align*}
where we have used  \eqref{GG-7}  \color{black} to deduce that 
\[
2^{-\theta k(\kappa_i -1)} \leq 2^{-\ell_ik}, \quad i=1,2,3.
\]
Then for any given $\epsilon>0$, we may first choose $\mu$ large such that 
\[
    \mu\inv S\, (2r_0)^\theta \, \sum_{k=1}^\infty a_k\inv f\left( 2r_0 \cdot 2^{-k} \right) <\frac{\epsilon}{2}.
\]
One may then choose a $t$ small enough such that also
\[
    t \cdot \sum_{i=1}^3 \left( \frac{C_i\, \mu\inv\cdot a_1\inv \cdot 2^{\ell_i+\theta}}{r_0^{\ell_i}} +  \frac{2^{\theta\kappa_i}\, \mu^{\kappa_i -1}}{r_0^{\ell_i}}\sum_{k=2}^\infty a_{k-1}^{\kappa_i}a_k\inv\right) < \frac{\epsilon}{2}.
\]
Whence we obtain for this $t$ small enough:
\[
    \mathbb P\left( G(t,R_0)>0\right) < \epsilon.
\]
\end{proof}

\appendix

\section{Justification of It\^o's formula in Lemma \ref{Lemma_localized_approx}}
\label{Appendix_Ito}
The aim of this section is to rigorously justify the It\^o expansion \eqref{Eq_Ito} of the localized energy of a weak solution to \eqref{Eq_viscous_reg}.
To this end, we continue under the assumptions of Lemma \ref{Lemma_localized_approx}, i.e.,  
 $u=u_{\kappa,l}$ is a weak solutions to \eqref{Eq_viscous_reg} in the sense of Definition \ref{def_weak_sol} with initial value $u_0$ satisfying \eqref{Eq_ass_u0} {and \eqref{assumption_m_plus_one}} and $\zeta\in C^2(\T^d)$. As announced there, following \cite[Lemma 8]{Lp_estimates}, we use the regularization 
 \begin{align}
     \vp_j(r) \,=\, \frac{1}{\alpha+1}\begin{cases}
         |r|^{\alpha+1 } , & |r| \le j, \\
        j^{\alpha-1}\bigl(
        \frac{\alpha(\alpha+1)}{2} r^2 \,-\, (\alpha+1)(\alpha - 1)j|r| \,+\, \frac{\alpha(\alpha-1)}{2} j^2
        \bigr),
         & |r|>j,
     \end{cases}
 \end{align}
 of $r^{\alpha+1}/(\alpha+1)$ with bounded and continuous second derivative, which makes the functional $\langle\vp_j(\cdot) ,\zeta^2\rangle$ admissible to It\^o's formula as in, e.g.,  \cite[Proposition A.1]{DHV_16}.
    {
    Regarding $u$, we observe that $\P$-a.s.\ \begin{equation}\label{Eq20} u \in C([0,T]; L^{\alpha+1}(\T^d))\qquad \text{and}\qquad  u,u^{(\alpha+m)/2} \in L^2([0,T]; H^1(\T^d)),\end{equation}
    by Definition \ref{def_weak_sol}. Let us also observe the subtle effect of the additionally imposed \eqref{assumption_m_plus_one} in the case $\alpha<m$: Only with this integrability may we apply \cite[Proposition 5.7]{Fehrman_Gess_ARMA} (with $p=m+1$) such that 
    \begin{align}\label{eqn_subtle}
        u^m \in L^2([0,T];H^1(\T^d)), 
    \end{align}
    $\P$-a.s., i.e., the minimal path regularity required to apply \cite[Proposition A.1]{DHV_16}. Since we also require finite second moments in order to apply \cite[Proposition A.1]{DHV_16}, we introduce the stopping times 
    \begin{equation}\label{Eq22}
        \tau_N\, =\,  \inf  \biggl\{ t\in [0, T ] \, \bigg| \, 
        \sup_{s\le t}\|u\|_{L^{\alpha+1}(\T^d)} \,+\, \int_{0}^t \int_{\T^d} 
        (1+ u^{\alpha+m-2} + u^{2m-2}) |\nabla u|^2
        \,\d x \,\d s  \,\ge\, N
        \biggr\} \,\wedge \, T,
 \end{equation}
 and record that it is sufficient  to verify \eqref{Eq_Ito} for  $t\le \tau_N$, since $\P (\tau_N = T)\nearrow 1$ as $N\ra +\infty$ by the above.
    }
    On $[0,\tau_N]$ we can however replace $u$ by $u_{\cdot \wedge \tau_N}$, so that also the assumptions from \cite[Proposition A.1]{DHV_16} on the stochastic process are satisfied and we deduce that, $\P$-a.s.,
    \begin{align}\begin{split}
    		\label{Eq_Ito_j}
    &	\int_{\T^d} \zeta^2 \vp_j(u_{t\wedge \tau_N} ) \,\d x \,-\, 	\int_{\T^d} \zeta^2 \vp_j(u_0 ) \,\d x \\&\quad =\, -m \int_0^{t\wedge \tau_N}\int_{\T^d} \zeta^2 \vp_j''(u)u^{m-1} |\nabla u |^2\,\d x \,\d s \,-\, 
    m \int_0^{t\wedge \tau_N} \int_{\T^d}  \vp_j'(u )u^{m-1} \nabla (\zeta^2) \cdot \nabla u \,\d x \,\d s
    \\&\qquad  -\kappa  \int_0^{t\wedge \tau_N}\int_{\T^d} \zeta^2 \vp_j''(u ) |\nabla u |^2\,\d x \,\d s \,-\, 
    \kappa  \int_0^{t\wedge \tau_N} \int_{\T^d}  \vp_j'(u) \nabla (\zeta^2) \cdot \nabla u \,\d x \,\d s
    \\&\qquad -\, \frac{1}{2} \int_0^{t\wedge \tau_N} \int_{\T^d} \zeta^2 \vp_j''(u)|\nabla u |^2\Psi_1 (\sigma_l')^2(u) \,\d x \,\d s \,-\, \frac{1}{2}
    \int_0^{t\wedge \tau_N} \int_{\T^d}  \vp_j'(u )\nabla (\zeta^2 )\cdot \nabla u \Psi_1 (\sigma_l')^2(u) \,\d x \,\d s
      \\&\qquad -\,\frac{1}{2}  \int_0^{t\wedge \tau_N} \int_{\T^d} \zeta^2 \vp_j''(u )\nabla u \cdot \Psi_2 (\sigma_l'\sigma_l )(u) \,\d x \,\d s \,-\, \frac{1}{2} 
     \int_0^{t\wedge \tau_N} \int_{\T^d}  \vp_j'(u )\nabla( \zeta^2 )\cdot  \Psi_2 (\sigma_l'\sigma_l )(u) \,\d x \,\d s
\\&\qquad 
+\,\frac{1}{2}\int_0^{t\wedge \tau_N} \int_{\T^d} \zeta^2 \vp_j''(u) 
(\sigma_l')^2(u) |\nabla u|^2 \Psi_1
 \,\d x \,\d s\,
 +\,\frac{1}{2}\int_0^{t\wedge \tau_N} \int_{\T^d} \zeta^2 \vp_j''(u) 
 (\sigma_l)^2(u) \Psi_3
 \,\d x \,\d s
 \\&\qquad 
 +\,{\int_0^{t\wedge \tau_N} \int_{\T^d} \zeta^2 \vp_j''(u) 
 (\sigma_l'\sigma_l)(u) \nabla u\cdot \Psi_2
 \,\d x \,\d s}
\\&\qquad 
-\,\sum_{k=1}^\infty \int_0^{t\wedge \tau_N} \int_{\T^d} \zeta^2 \vp_j'(u)   \sigma_l'(u)  \psi_k \nabla u \,\d x\cdot \d\beta_k  \,-\,\sum_{k=1}^\infty \int_0^{t\wedge \tau_N} \int_{\T^d} \zeta^2 \vp_j'(u)   \sigma_l(u) \nabla \psi_k  \,\d x\cdot \d\beta_k 
 \\&\quad =\, A_1^{(j)}+\,\dots \,+ A_4^{(j)} \,+\, B_1^{(j)}+\,\dots \,+ B_7^{(j)}
\,+\, C_1^{(j)}+ C_2^{(j)},
\end{split}
\end{align}
for $t \in [0, T]$. It remains to argue that for fixed $t$, the terms in the above equality converge $\P$-a.s. to their counterpart in \eqref{Eq_Ito}, where $\vp_j$ is replaced by the desired $r^{\alpha+1}/(\alpha+1)$. Indeed, since then both sides of the identity define continuous processes, also equality for all $t\in [0,T]$, $\P$-a.s., follows.
To proceed, additionally to the path properties \eqref{Eq20}, we recall also that $u$ is non-negative, that $
	\vp_j(r) = r^{\alpha+1}/(\alpha+1)$
for $r \in [0, j]$, and the bounds
\begin{equation}\label{Eq21}
	|\vp_j(r)| \lesssim_\alpha |r|^{\alpha+1} ,\qquad 
		|\vp_j'(r)| \lesssim_\alpha |r|^{\alpha} ,\qquad 
			|\vp_j''(r)| \lesssim_\alpha |r|^{\alpha-1} ,
\end{equation}
independent of $j$. In the following, we take $j\to\infty$ in the various terms in \eqref{Eq_Ito_j}, separately.

\emph{The left-hand side.} We have that $u_{t\wedge\tau_N} \in L^{\alpha+1}(\T^d)$,  $\P$-a.s., due to \eqref{Eq20}, as well as $\vp_j(u_{t\wedge\tau_N}) \to u_{t\wedge\tau_N}^{\alpha+1}/(\alpha+1)$, $\P \otimes \d x$-almost everywhere, as $j\to\infty$. Invoking also \eqref{Eq21}, it follows that 
$\P$-a.s.
\[
\int_{\T^d} \zeta^2 \vp_j(u_{t\wedge \tau_N} ) \,\d x  \,\to \, 
\frac{1}{\alpha+1}\int_{\T^d} \zeta^2 u_{t\wedge \tau_N}^{\alpha+1} \,\d x ,
\] from the dominated convergence theorem and the same reasoning applies to the term involving  $u_0$.

\emph{Ad $A$.} These limits can be taken based on dominated convergence too, which we demonstrate exemplarily on $A^{(j)}_2$: An application is justified by the $\P\otimes \d t \otimes  \d x$-a.e.\ convergence $ \vp_j' (u) \to u^\alpha$, the second estimate from \eqref{Eq21} and the bound 
\[
\int_0^{t\wedge \tau_N} \int_{\T^d} |u^{\alpha +m-1}\nabla (\zeta^2 )\cdot  \nabla u| \d x \d s \,\lesssim_{\zeta} \|u^{(\alpha+m)/2}\|_{L^2([0,T]\times \T^d)}  \|\nabla u^{(\alpha+m)/2}\|_{L^2([0,T]\times \T^d)}, 
\]
which is $\P$-a.s.\ finite by \eqref{Eq20}.

\emph{Ad $B$.} We can use the same arguments as for the $A$ terms, invoking additionally the uniform  boundedness of $\Psi_1$, $\Psi_2$ and $\Psi_3$ as well as $\sigma_l$ and $\sigma_l'$ for fixed $l$, implied by \eqref{Eq_ass_B2} and \eqref{Eq_ass_sigma_l}.

\emph{Ad $C$.} As all remaining terms of the identity \eqref{Eq_Ito_j} converge for fixed $t$, $\P$-a.s., the same must hold for the sum $C_1^{(j)} + C_2^{(j)}$. Since an almost sure and a limit in probability coincide, it suffices to argue that
\begin{align}
	&\label{Eq23}
	\sum_{k=1}^\infty \int_0^{\tau_N}\biggl| \int_{\T^d} \zeta^2 \bigl(\vp_j'(u)-u^\alpha\bigr) \sigma_l'(u)  \psi_k \nabla u  \,\d x\biggr|^2 \d s
	 \,\to\, 0,\\
		&
	\sum_{k=1}^\infty \int_0^{\tau_N}\biggl| \int_{\T^d} \zeta^2 \bigl(\vp_j'(u)-u^\alpha\bigr)  \sigma_l(u) \nabla \psi_k \,\d x\biggr|^2 \d s \,\to\, 0, \label{Eq25}
\end{align}
in probability, or even $\P$-almost surely. For the first term  we integrate by parts to rewrite it as
\begin{align} \label{Eq24}
		\sum_{k=1}^\infty \int_0^{\tau_N}\biggl| \int_{\T^d} \Theta_j (u) \nabla (\zeta^2 \psi_k)    \,\d x\biggr|^2 \d s\,\le\, \biggl(\sup_{s\le \tau_N}	\sum_{k=1}^\infty  \int_{\T^d} |\Theta_j (u)| |\nabla (\zeta^2 \psi_k) |^2   \,\d x \biggr) \int_0^{\tau_N} \int_{\T^d} |\Theta_j (u)|    \,\d x\, \d s,
\end{align}
where $\Theta_j(u) = \int_0^u  (\vp_j'(r) - r^\alpha  ) \sigma_l'(r) \d r$, and we used H\"older's inequality in the estimate. Regarding the prefactor on the right-hand side, we employ $|\Theta_j(r)|\lesssim_{\alpha,l} |r|^{\alpha+1}$ by \eqref{Eq_ass_sigma_l} and \eqref{Eq21} together with \eqref{Eq_ass_B2} and the definition \eqref{Eq22} of $\tau_N$ to deduce that
\[ \sup_{s\le \tau_N} 
\sum_{k=1}^\infty\int_{\T^d} |\Theta_j (u)| |\nabla (\zeta^2 \psi_k) |^2   \,\d x \,\lesssim_{\alpha,l, \psi}\, \sup_{s\le \tau_N} \|u\|_{L^{\alpha+1}(\T^d)}^{\alpha+1} \,\le\, N^{\alpha+1}.
\]
Meanwhile, we have $\Theta_j (r) = 0$ for $r \in [0,j]$ allowing us to conclude using dominated convergence that
\[
\int_0^{\tau_N} \int_{\T^d} |\Theta_j (u)|    \,\d x\, \d s \,\to\, 0,
\]
$\P$-a.s., and therefore also \eqref{Eq23}. The convergence \eqref{Eq25} can be shown analogously, which finishes the proof.

\section{A Two-Step Iteration Lemma}
\label{App_B}
We will use this section to prove an important iteration lemma, which is used to establish Lemma \ref{lemma-iteration-expectation}. The result is quite similar in spirit to some other iteration lemmas that have found usage in free-boundary analysis (see for instance \cite[Lemma 4]{Shishkov_Hulshof_FSOP}). In contrast to the aforementioned arguments however, the succeeding lemma will need to employ more than one recursive looping argument. This is a bit technical, but leads to a more straightforward proof of Theorem \ref{Thm_fsop}. 

\begin{lemma}[Iteration Trick]\label{lemma-iteration-trick} Let $R>0$, $N,M\in \N$. Let $V,E, A, \left\{ U_i \right\}_{i=1}^N, \left\{F_j\right\}_{j=1}^M$ be non-negative, locally bounded functions defined on $[0,R)$ such that $A,F_j$ are non-increasing for all $j=1,\dots, M$. We assume that there exist $L_i,c_i, \tilde c, \ell_j\in [0,\infty)$, such that for any $s\ge 0$ and $\delta>0$ with $s+\delta \leq R$ we have  
    \begin{equation}
        \label{eq-iteration-condition-general}
        \begin{split}
        V(s+\delta) + E(s+\delta)& \leq A(s) + \sum_{i=1}^N \left(\frac{2}{\delta}\right)^{L_i} U_i\left(s+\frac{\delta}{2}\right) \\
        & \leq A(s) + \epsilon\sum_{i=1}^N c_i\left(\frac{4}{\delta}\right)^{L_i}U_i\left(s + \frac{\delta}{4} \right) + \epsilon \, \tilde c \, E\left(s\right) + C_\epsilon\sum_{j=1}^M \left( \frac{2}{\delta}\right)^{\ell_j} F_j(s),
        \end{split}
    \end{equation}for any $\epsilon>0$ and a constant $C_\epsilon<\infty$. Then, there exists a  constant $C^*\in (0,\infty)$ depending only on the preceding parameters, such that for any $s,\delta$:
    \begin{equation}
        \label{eq-iteration-lemma-result}
        V(s+\delta) \leq C^* \left(A(s) + \sum_{j=1}^M\delta^{-\ell_j} F_j(s)\right).
    \end{equation}
\end{lemma}

\begin{proof}
    As previously mentioned, the proof takes advantage of two recursive loops. To begin the first, we first estimate the second line of \eqref{eq-iteration-condition-general} using 
    \[
    c_i\leq c :=\max_i c_i.
    \]
    Since the estimate \eqref{eq-iteration-condition-general} remains when increasing the involved constants, we may assume in the following that  $\tilde c,c>0$. If this were not the case, one may leave out one recursive argument, or if even $c=\tilde c=0$, there is nothing to be done. 
    Observe now that the terms 
    \[
        \epsilon c \sum_{i=1}^N\left(\frac{4}{\delta}\right)^{L_i}U_i\left(s + \frac{\delta}{4} \right)
    \]
    which appear in the second line of \eqref{eq-iteration-condition-general}, are nothing more than the same terms that appear in the right side of the first line with $\delta/2$ replaced with $\delta/4$ (and multiplied by $\epsilon c$). Using the second inequality of \eqref{eq-iteration-condition-general} implies for each $i$:
    \[
        \sum_{i=1}^N\epsilon c \left(\frac{4}{\delta}\right)^{L_i}U_i\left(s + \frac{\delta}{4} \right) \leq \epsilon c \left( A(s) + \epsilon\sum_{i=1}^N c\left(\frac{8}{\delta}\right)^{L_i}U_i\left(s + \frac{\delta}{8} \right) + \epsilon \, \tilde c \, E\left(s\right) + C_\epsilon\sum_{j=1}^M \left( \frac{4}{\delta}\right)^{\ell_j} F_j(s) \right).
    \]
    Substituting this into the last line of \eqref{eq-iteration-condition-general} yields 
    \begin{equation}
        \label{eq-iteration-general-1}
        \begin{split}
        V(s+\delta) + E(s+\delta)& \leq A(s) + \sum_{i=1}^N \left(\frac{2}{\delta}\right)^{L_i} U_i\left(s+\frac{\delta}{2}\right) \\
        & \leq (1+c\epsilon)A(s) + \epsilon^2c^2\sum_{i=1}^N \left(\frac{8}{\delta}\right)^{L_i}U_i\left(s + \frac{\delta}{8} \right) + \epsilon\tilde c(1+c\epsilon) \, E\left(s\right)  \\
        & \hspace{1cm} + C_\epsilon\sum_{j=1}^M \left(\frac{2}{\delta}\right)^{\ell_j}  \left( 1+c\epsilon 2^{\ell_j}\right)  F_j(s).
        \end{split}
    \end{equation}
    Suppose now for some $k\in \N$ that
    \begin{equation}
        \label{eq-iteration-general-2}
        \begin{split}
        V(s+\delta) + E(s+\delta)& \leq A(s) + \sum_{i=1}^N \left(\frac{2}{\delta}\right)^{L_i} U_i\left(s+\frac{\delta}{2}\right) \\
        & \leq \left(\sum_{n=1}^k (c\epsilon)^{n-1}\right)A(s) + \sum_{i=1}^N (2^{L_i} c\epsilon)^k \left(\frac{2}{\delta}\right)^{L_i}U_i\left(s + \frac{\delta}{2^{k+1}} \right) + \epsilon\tilde c\left( \sum_{n=1}^k(\epsilon c)^{n-1}\right) \, E\left(s\right)  \\
        & \hspace{1cm} + C_\epsilon\sum_{j=1}^M \left(\frac{2}{\delta}\right)^{\ell_j}\left( \sum_{n=1}^k (2^{\ell_j}c\epsilon)^{n-1}\right)  F_j(s).
        \end{split}
    \end{equation}
    Using precisely the same logic as before 
    \begin{align*}
        \sum_{i=1}^N\epsilon^k c^k \left(\frac{2^{k+1}}{\delta}\right)^{L_i} & U_i\left(s + \frac{\delta}{2^{k+1}} \right) \\
        & \leq \epsilon^k c^k \left( A(s) + \epsilon\sum_{i=1}^N c\left(\frac{2^{k+2}}{\delta}\right)^{L_i}U_i\left(s + \frac{\delta}{2^{k+2}} \right) + \epsilon \, \tilde c \, E\left(s\right) + C_\epsilon\sum_{j=1}^M \left( \frac{2^{k+1}}{\delta}\right)^{\ell_j} F_j(s) \right).
    \end{align*}
    Substituting this back into the induction step \eqref{eq-iteration-general-2} and combining like terms completes the induction proof. Therefore, \eqref{eq-iteration-general-2} holds for all $k\in \N$. For any $\epsilon>0$ chosen small enough, specifically if 
    \[
        \epsilon < \min_{i,j}\left\{c\inv, (2^{\ell_j}c)\inv, (2^{L_i}c)\inv\right\},
    \]
    we may pass to the limit $k\ra \infty$. Note we also use for this that each $U_i$ is locally bounded. Passing to the limit, the first recursion is complete, and we obtain
    \begin{equation}
        \label{eq-iteration-general-3}
        \begin{split}
        V(s+\delta) + E(s+\delta)& \leq A(s) + \sum_{i=1}^N \left(\frac{2}{\delta}\right)^{L_i} U_i\left(s+\frac{\delta}{2}\right) \\
        & \leq S_1 \,A(s) + \epsilon\tilde c \, S_2 \, E\left(s\right)  + C_\epsilon\sum_{j=1}^M \left(\frac{2}{\delta}\right)^{\ell_j}  S_{3,j}F_j(s).
        \end{split}
    \end{equation}
    Here the values of $S_1 = S_1(\epsilon)$, $S_2=S_2(\epsilon)$ and $S_{3,j}=S_{3,j}(\epsilon)$ are the values of the geometric series that are obtained in the limit $k\ra \infty$, and are independent of $s$ and $\delta$. An additional important note is that the value of each of these series converges to one whenever $\epsilon\ra 0$. That is to say, we may choose $\epsilon$ even smaller without increasing  these terms. 

    We now ignore the middle term in \eqref{eq-iteration-general-3}, focusing only on the first and the last: 
    \[
        V(s+\delta) + E(s+\delta) \leq S_1 \,A(s) + \epsilon\tilde c \, S_2 \, E\left(s\right)  + C_\epsilon\sum_{j=1}^M \left(\frac{2}{\delta}\right)^{\ell_j}  S_{3,j}  F_j(s).
    \]
    Since $s$ and $\delta$ are arbitrary, we may moreover  reset $s' = s+\frac{\delta}{2}$, $\delta' = \frac{\delta}{2}$. With these values, the above inequality becomes 
    \begin{equation}
        \label{eq-iteration-general-4}     V(s+\delta) + E(s+\delta) \leq S_1 \,A\left( s+\frac{\delta}{2}\right) + \epsilon\tilde c \, S_2 \, E\left(s+\frac{\delta}{2}\right)  + C_\epsilon\sum_{j=1}^M \left(\frac{4}{\delta}\right)^{\ell_j}  S_{3,j}  F_j\left(s+\frac{\delta}{2}\right),
    \end{equation}
    since again $s' +\delta' =s+\delta$.
    Now we perform the same recursion as above in the $E$-terms. Indeed, replacing $\delta$ by $\delta/2$, the above estimate implies that 
    \[
        \epsilon \tilde c \, S_2 \, E\left(s+\frac{\delta}{2}\right) \leq \epsilon \tilde c \, S_2 \left(  S_1 \,A\left( s+\frac{\delta}{4}\right) + \epsilon\tilde c \, S_2 \, E\left(s+\frac{\delta}{4}\right)  + C_\epsilon\sum_{j=1}^M \left(\frac{8}{\delta}\right)^{\ell_j}  S_{3,j}  F_j\left(s+\frac{\delta}{4}\right) \right).
    \]
    Substituting this into \eqref{eq-iteration-general-4}, one obtains 
    \begin{equation}
        \label{eq-iteration-general-5}
        V(s+\delta) + E(s+\delta) \leq S_1(1+\epsilon \tilde c S_2) \,A\left( s\right) + \epsilon^2\tilde c^2 \, S_2^2 \, E\left(s+\frac{\delta}{4}\right)  + C_\epsilon\sum_{j=1}^M \left(\frac{4}{\delta}\right)^{\ell_j}  S_{3,j} \left( 1+ \epsilon \, \tilde c \, S_2 \, 2^{\ell_j}\right)  F_j\left(s\right),
    \end{equation}
    using that $A,F_j$ are non-increasing by assumption. Repeating  $k-2$ times yields:
     \begin{equation}
        V(s+\delta) + E(s+\delta) \leq S_1\left( \sum_{n=1}^k (\epsilon \tilde c S_2)^{n-1}\right) \,A\left( s\right) + \epsilon^k\tilde c^k \, S_2^k \, E\left(s+\frac{\delta}{2^k}\right)  + C_\epsilon\sum_{j=1}^M \left(\frac{4}{\delta}\right)^{\ell_j}  S_{3,j} \left( \sum_{n=1}^k (\epsilon \tilde c 2^{\ell_j} S_2)^{n-1} \right)  F_j\left(s\right).
    \end{equation}
    We see once again that by choosing $\epsilon$ small enough, we may pass to the limit $k\ra \infty$, from which the result follows by absorbing the  remaining coefficients in $C^*$. 
    
\end{proof}

\section{Properties of $\kappa_i$, $\ell_i$}
\label{App_C}
In this section, we show some results of calculus which relate the parameters $p_i, \kappa_i, \ell_i, \theta_i, $ and $\theta$ to each other. These are essential ingredients in the proofs of Theorems~\ref{Thm_fsop} and \ref{thm-fsop-quant} and of Lemma~\ref{lemma-iteration-expectation}.
\begin{lemma}
    \label{lemma-algebra}
    The parameter
    \[
    \theta =  \frac{2(\alpha+1)+d(m-1)}{2(n-1)}, 
    \]
    introduced in \eqref{GG-999} in the statement of Theorem \ref{Thm_fsop}, and the parameters  $\kappa_i$, $\ell_i$, $i\in\{1,2,3\}$, introduced in  \eqref{eq-iteration-lemma-setup} and in \eqref{eq-def-kappa},   satisfy the following relations:
    \begin{enumerate}[label=(\roman*)]
    \item $\kappa_i>1$ for all $i\in  \{1,2,3\}$,
    \item$\theta=\max_{i\in\{1,2,3\}}\left\{\frac{\ell_i}{\kappa_i -1}\right\}$.
        \end{enumerate}
   
\end{lemma}
    
    \begin{proof}
      To prove i), we recall from \eqref{GG-3} and \eqref{eq-def-kappa}  that  
        \[
            \kappa_ i := \frac{p_i(1-\theta_i)(\alpha+m)}{(\alpha+1)(\alpha+m - p_i\theta_i)},
        \]
        where 
        \[
        p=(p_1,p_2,p_3):=(\alpha+m, \alpha+2n-1, \alpha+2n-1)
        \]
        and that
        \[
            \theta_i := \frac{\frac{\alpha+m}{2(\alpha+1)} - \frac{\alpha+m}{2p_i}}{\frac{1}{d}-\frac{1}{2}+\frac{\alpha+m}{2(\alpha+1)}} = \frac{d(\alpha+m)(\alpha+1)\left(\frac{1}{\alpha+1} -\frac{1}{p_i}\right)}{2(\alpha+1)+d(m-1)}.
        \]
        We first simplify $\alpha+m-p_i\theta_i$, obtaining 
        \[
            \alpha+m-p_i\theta_i = 
            (\alpha+m) \left(\frac{2(\alpha+1) +d(m-1) -d(p_i-\alpha-1) }{2(\alpha+1) + d(m-1)} \right)=(\alpha+m)\left(\frac{2(\alpha+1)+d(\alpha+m-p_i)}{2(\alpha+1) + d(m-1)} \right).
        \]
        This allows us to compute then
        \[
            \frac{p_i(\alpha+m)}{\alpha+m-p_i\theta_i} = \frac{p_i(2(\alpha+1)+d(m-1))}{2(\alpha+1) + d(\alpha+m-p_i)}.
        \]
        We also see that 
        \[
            1-\theta_i = \frac{2(\alpha+1) +d(m-1)-d(\alpha+m)\left( 1- \frac{\alpha+1}{p_i} \right) }{2(\alpha+1) + d(m-1)}.
        \]
        Combining all of the above simplifications in the expression for $\kappa_i$, we obtain
        \begin{align}\label{GG-2}\begin{split}
            \kappa_i & = \frac{p_i}{\alpha+1} \left(\frac{2(\alpha+1) +d(m-1)-d(\alpha+m)\left( 1- \frac{\alpha+1}{p_i} \right) }{2(\alpha+1) + d(\alpha+m-p_i)}\right) \\ & = \frac{2p_i(\alpha+1) + d(-p_i(\alpha+1) + (\alpha+m)(\alpha+1))}{(\alpha+1)(2(\alpha+1) + d(\alpha+m-p_i))} \\
            &= \frac{2p_i + d(\alpha+m-p_i)}{2(\alpha+1) + d(\alpha+m-p_i)}.
            \end{split}
        \end{align}
        Due to \eqref{equation_n_m_together}, we have $p_i \in (\alpha+1, \alpha+m)$ for any $i$, and so $\kappa_i>1$. The reader might also note that these values of $\kappa_i$ are consistent with those given in the statement of Theorem \ref{Thm_fsop}, upon inserting $p_i$. 

        For the second part, recall that 
        \[
        L=(L_1,L_2,L_3):= (2,2,0)
        \]
        and  that
        \begin{equation}\label{GG-1}
            \ell_i := \frac{L_i(\alpha+m)}{\alpha+m-p_i\theta_i} = \frac{L_i(2(\alpha+1)+d(m-1))}{2(\alpha+1) + d(\alpha+m-p_i)},
        \end{equation}
        where we have used the previous formula for $\alpha+m-p_i\theta_i$. Then we use the above expression for $\kappa_i$ to see 
        \[
            \frac{\ell_i}{\kappa_i -1} = \frac{L_i(2(\alpha+1)+d(m-1))}{2p_i - 2(\alpha+1)}.
        \]
        Once again due to \eqref{equation_n_m_together}, we have $p_i \geq \alpha+2n-1$, which implies that the largest possible value of the above number is 
        \[
            \frac{\ell_2}{\kappa_2-1} =\frac{ 2(\alpha+1)+d(m-1)}{2(n-1)},
        \]
        which completes the proof. 
        
      \end{proof}

\section{A measurability argument}

In this section, we show that the random times $\tau_k$ which were defined in \eqref{def_tau_k} are in fact stopping times, as it has been stated in the proof of Theorem~\ref{thm-fsop-quant}.
Instead of using the quite elaborate d\'ebut theorem, we present  an elementary proof here. 
\begin{lemma}
    \label{Lemma_appendix_D} Let $u$ be the unique kinetic solution to \eqref{Eq_SPME_intro} with initial value $u_0$. For a given $x_0\in \mathbb T^d\setminus\mathrm{supp}(u_0)$ define the process $G:[0,T] \times \Omega\times \left[0,\frac{1}{2}\right) \ra \R_{\geq 0}$ by 
    \[
        G(t,\omega, r): = \sup_{s\in (0,t)} \, \int_{B_r(x_0)} u_s^{\alpha+1}(\omega).
    \]
    Let $(\mu,r)\in \R_+\times \left[0,\frac{1}{2}\right)$. Then for any $t\in (0,T)$, the random time 
    \[
        \tau:= \inf\left\{ s>0, \, G(s,r)>\mu\right\}
        \wedge t
    \]
    is a stopping time. 
\end{lemma}
    \begin{proof}Below we suppress the dependence of the involved processes and stopping times on $\omega$ and $r$. For $\delta\in (0,1)$, 
    	we take $f_\delta\in C_0^\infty(\R)$  with $f_\delta(y)\nearrow |y|^{\alpha+1}$ as $\delta\searrow 0$ for all $y\in \R$. In particular, the processes 
    	\[
    	G_\delta(s): = \sup_{s'\in (0,s)} \, \int_{B_r(x_0)} f_\delta (u_{s'}
)    	\]
    	are then $\P$-a.s. continuous in time
    	by the continuity of $u(t,\cdot,\omega)$ as a mapping from $[0,T]$ into $L^1(\T^d)$.
         Moreover, the above choice guarantees  $G_\delta(s)\nearrow G(s)$ for all $s\in [0,T]$.  Following \cite[Section~1.2]{KaratzasShreve}, this  ensures that 
    	\[
    	\tau_\delta:= \inf\left\{ s>0, \, G_\delta(s)>\mu\right\}
        \wedge t
    	\] 
    	is a stopping time and moreover
    	that $\tau_{\delta'} \ge \tau_\delta \ge \tau$ for $\delta'>\delta >0$. If $\tau=t$ we have therefore $\tau_\delta =\tau$ and otherwise we may take for  $\epsilon>0$ a time instant $t_* \in [\tau, \tau+\epsilon]$ such that $G(t_*)>\mu$. The convergence $G_\delta(s)\nearrow G(s)$ ensures that for all sufficiently small $\delta>0$ also $G_\delta(t_*) >\mu$, i.e., that $\tau_\delta \le t_*$. Since $\epsilon$ was arbitrary we deduce that $\tau_\delta \searrow \tau$. 
    	
    	We have therefore in both cases
    	\[
		\{\tau< s\} \,\subset\, \bigcup_{\delta>0} \{\tau_\delta <s\} \,\in \,\mathscr{F}_s    	
    	\]
    	for each $s\in [0,t]$.
        On the other hand, $\tau(\omega)\geq s$ gives $\tau_\delta(\omega)\geq s$
        which implies $\{\tau\geq s\}\cap \bigcup_{\delta>0}\{\tau_\delta<s\} = \emptyset$ and therefore the opposite inclusion, too. Hence,
    	$\tau$ is  optional. The right-continuity of $\mathscr{F}$ gives the claim. 
    \end{proof}

\noindent
\textbf{Acknowledgements.}
J.U.\ has been supported by the Graduiertenkolleg 2339 IntComSin
”Interfaces, Complex Structures, and Singular Limits” of the Deutsche Forschungsgemeinschaft (DFG, German Research Foundation) with Project-ID 321821685. The support
is gratefully acknowledged. M.S. would like to thank the FAU Erlangen-Nürnberg and the Graduiertenkolleg 2339 for their hospitality during several stays in Erlangen in order to work on this project. 

\vspace{.2cm}
\noindent
\textbf{Data availability.} This manuscript has no associated data.

\vspace{.2cm}
\noindent
\textbf{Declaration – Conflict of interest.} The authors have no conflict of interest.

\bibliographystyle{abbrv}
\bibliography{SPME_bib_n}

\end{document}